\documentclass{amsart}
\usepackage{amssymb,amsmath,color}
\usepackage{dsfont}
\usepackage{mathrsfs}
\usepackage{enumitem}
\newcommand{\RR}{\mathbb{R}}

\usepackage{color}
\usepackage{tikz}
\usepackage{pgf}
\usepackage{genyoungtabtikz}
\usepackage{bm}

\makeatletter
\newcommand{\addresseshere}{%
  \enddoc@text\let\enddoc@text\relax
}
\makeatother

\newcommand{\cI}{\mathcal{I}}
\newcommand{\cV}{\mathcal{V}}

\newcommand{\fS}{\mathfrak{S}}
\newcommand{\fs}{\mathfrak{s}}
\newcommand{\fR}{\mathfrak{R}}
\newcommand{\sym}{\textup{sym}}

\newtheorem{theorem}{Theorem}[section]

\newtheorem{lemma}[theorem]{Lemma}
\newtheorem{corollary}[theorem]{Corollary}
\newtheorem{proposition}[theorem]{Proposition}
\theoremstyle{definition}
\newtheorem{definition}[theorem]{Definition}
\newtheorem{example}[theorem]{Example}

\theoremstyle{remark}
\newtheorem{remark}[theorem]{Remark}

\newcommand{\labeledHzero}[3]{%top label, bottom left label, bottom right label}
\resizebox{0.8cm}{0.4cm}{
\begin{tikzpicture}
\node[fill=black, circle, minimum size=0.3cm, label=right:\Huge{\textbf{#1}}] (1) {};
\node[fill=black, circle, minimum size=0.3cm, label=left:\Huge{\textbf{#2}}] (2) [below left of=1] {};
\node[fill=black, circle, minimum size=0.3cm, label=right:\Huge{\textbf{#3}}] (3) [below right of=1] {};
\end{tikzpicture}}}

\newcommand{\labeledHone}[3]{%top label, bottom left label, bottom right label}
\resizebox{0.8cm}{0.4cm}{
\begin{tikzpicture}
\node[fill=black, circle, minimum size=0.3cm, label=right:\Huge{\textbf{#1}}] (1) {};
\node[fill=black, circle, minimum size=0.3cm, label=left:\Huge{\textbf{#2}}] (2) [below left of=1] {};
\node[fill=black, circle, minimum size=0.3cm, label=right:\Huge{\textbf{#3}}] (3) [below right of=1] {};
\draw (2)--(1);
\end{tikzpicture}}}

\newcommand{\labeledHoneright}[3]{%top label, bottom left label, bottom right label}
\resizebox{0.8cm}{0.4cm}{
\begin{tikzpicture}
\node[fill=black, circle, minimum size=0.3cm, label=right:\Huge{\textbf{#1}}] (1) {};
\node[fill=black, circle, minimum size=0.3cm, label=left:\Huge{\textbf{#2}}] (2) [below left of=1] {};
\node[fill=black, circle, minimum size=0.3cm, label=right:\Huge{\textbf{#3}}] (3) [below right of=1] {};
\draw (3)--(1);
\end{tikzpicture}}}

\newcommand{\labeledHonebottom}[3]{%top label, bottom left label, bottom right label}
\resizebox{0.8cm}{0.4cm}{
\begin{tikzpicture}
\node[fill=black, circle, minimum size=0.3cm, label=right:\Huge{\textbf{#1}}] (1) {};
\node[fill=black, circle, minimum size=0.3cm, label=left:\Huge{\textbf{#2}}] (2) [below left of=1] {};
\node[fill=black, circle, minimum size=0.3cm, label=right:\Huge{\textbf{#3}}] (3) [below right of=1] {};
\draw (3)--(2);
\end{tikzpicture}}}

\newcommand{\labeledHtwo}[3]{%top label, bottom left label, bottom right label}
\resizebox{0.8cm}{0.4cm}{
\begin{tikzpicture}
\node[fill=black, circle, minimum size=0.3cm, label=right:\Huge{\textbf{#1}}] (1) {};
\node[fill=black, circle, minimum size=0.3cm, label=left:\Huge{\textbf{#2}}] (2) [below left of=1] {};
\node[fill=black, circle, minimum size=0.3cm, label=right:\Huge{\textbf{#3}}] (3) [below right of=1] {};
\draw (3)--(1)--(2);
\end{tikzpicture}}}

\newcommand{\labeledHtwoleft}[3]{%top label, bottom left label, bottom right label}
\resizebox{0.8cm}{0.4cm}{
\begin{tikzpicture}
\node[fill=black, circle, minimum size=0.3cm, label=right:\Huge{\textbf{#1}}] (1) {};
\node[fill=black, circle, minimum size=0.3cm, label=left:\Huge{\textbf{#2}}] (2) [below left of=1] {};
\node[fill=black, circle, minimum size=0.3cm, label=right:\Huge{\textbf{#3}}] (3) [below right of=1] {};
\draw (1)--(2)--(3);
\end{tikzpicture}}}

\newcommand{\labeledHtworight}[3]{%top label, bottom left label, bottom right label}
\resizebox{0.8cm}{0.4cm}{
\begin{tikzpicture}
\node[fill=black, circle, minimum size=0.3cm, label=right:\Huge{\textbf{#1}}] (1) {};
\node[fill=black, circle, minimum size=0.3cm, label=left:\Huge{\textbf{#2}}] (2) [below left of=1] {};
\node[fill=black, circle, minimum size=0.3cm, label=right:\Huge{\textbf{#3}}] (3) [below right of=1] {};
\draw (1)--(3)--(2);
\end{tikzpicture}}}

\newcommand{\labeledcherries}[3]{%top label, bottom left label, bottom right label}
\resizebox{0.9cm}{0.4cm}{
\begin{tikzpicture}
\node[fill=black, circle, minimum size=0.3cm, label=right:\Huge{\textbf{#1}}] (1) {};
\node[fill=black, circle, minimum size=0.3cm, label=left:\Huge{\textbf{#2}}] (2) [below left of=1] {};
\node[fill=black, circle, minimum size=0.3cm, label=right:\Huge{\textbf{#3}}] (3) [below right of=1] {};
\draw (2)--(1)--(3);
\end{tikzpicture}}}

\newcommand{\flagzero}[4]{%top label, bottom left label, bottom right label, bottomest label}
\resizebox{0.8cm}{0.4cm}{
\begin{tikzpicture}
\node[fill=black, circle, minimum size=0.3cm, label=right:\Huge{\textbf{#1}}] (1) {};
\node[fill=black, circle, minimum size=0.3cm, label=left:\Huge{\textbf{#2}}] (2) [below left of=1] {};
\node[fill=black, circle, minimum size=0.3cm, label=right:\Huge{\textbf{#3}}] (3) [below right of=1] {};
\node[fill=black, circle, minimum size=0.3cm, label=right:\Huge{\textbf{#4}}] (4) [below left of=3] {};
\draw (2)--(1)--(3);
\end{tikzpicture}}}

\newcommand{\gzero}[4]{%top label, bottom left label, bottom right label, bottomest label}
\resizebox{0.8cm}{0.4cm}{
\begin{tikzpicture}
\node[fill=black, circle, minimum size=0.3cm, label=right:\Huge{\textbf{#1}}] (1) {};
\node[fill=black, circle, minimum size=0.3cm, label=left:\Huge{\textbf{#2}}] (2) [below left of=1] {};
\node[fill=black, circle, minimum size=0.3cm, label=right:\Huge{\textbf{#3}}] (3) [below right of=1] {};
\node[fill=black, circle, minimum size=0.3cm, label=right:\Huge{\textbf{#4}}] (4) [below left of=3] {};
\draw (3)--(2)--(1)--(3);
\end{tikzpicture}}}

\newcommand{\flagone}[4]{%top label, bottom left label, bottom right label, bottomest label}
\resizebox{0.8cm}{0.4cm}{
\begin{tikzpicture}
\node[fill=black, circle, minimum size=0.3cm, label=right:\Huge{\textbf{#1}}] (1) {};
\node[fill=black, circle, minimum size=0.3cm, label=left:\Huge{\textbf{#2}}] (2) [below left of=1] {};
\node[fill=black, circle, minimum size=0.3cm, label=right:\Huge{\textbf{#3}}] (3) [below right of=1] {};
\node[fill=black, circle, minimum size=0.3cm, label=right:\Huge{\textbf{#4}}] (4) [below left of=3] {};
\draw (2)--(1)--(3);
\draw (1)--(4);
\end{tikzpicture}}}

\newcommand{\gone}[4]{%top label, bottom left label, bottom right label, bottomest label}
\resizebox{0.8cm}{0.4cm}{
\begin{tikzpicture}
\node[fill=black, circle, minimum size=0.3cm, label=right:\Huge{\textbf{#1}}] (1) {};
\node[fill=black, circle, minimum size=0.3cm, label=left:\Huge{\textbf{#2}}] (2) [below left of=1] {};
\node[fill=black, circle, minimum size=0.3cm, label=right:\Huge{\textbf{#3}}] (3) [below right of=1] {};
\node[fill=black, circle, minimum size=0.3cm, label=right:\Huge{\textbf{#4}}] (4) [below left of=3] {};
\draw (3)--(2)--(1)--(3);
\draw (1)--(4);
\end{tikzpicture}}}

\newcommand{\flagtwo}[4]{%top label, bottom left label, bottom right label, bottomest label}
\resizebox{0.8cm}{0.4cm}{
\begin{tikzpicture}
\node[fill=black, circle, minimum size=0.3cm, label=right:\Huge{\textbf{#1}}] (1) {};
\node[fill=black, circle, minimum size=0.3cm, label=left:\Huge{\textbf{#2}}] (2) [below left of=1] {};
\node[fill=black, circle, minimum size=0.3cm, label=right:\Huge{\textbf{#3}}] (3) [below right of=1] {};
\node[fill=black, circle, minimum size=0.3cm, label=right:\Huge{\textbf{#4}}] (4) [below left of=3] {};
\draw (4)--(2)--(1)--(3);
\end{tikzpicture}}}

\newcommand{\gtwo}[4]{%top label, bottom left label, bottom right label, bottomest label}
\resizebox{0.8cm}{0.4cm}{
\begin{tikzpicture}
\node[fill=black, circle, minimum size=0.3cm, label=right:\Huge{\textbf{#1}}] (1) {};
\node[fill=black, circle, minimum size=0.3cm, label=left:\Huge{\textbf{#2}}] (2) [below left of=1] {};
\node[fill=black, circle, minimum size=0.3cm, label=right:\Huge{\textbf{#3}}] (3) [below right of=1] {};
\node[fill=black, circle, minimum size=0.3cm, label=right:\Huge{\textbf{#4}}] (4) [below left of=3] {};
\draw (4)--(2)--(1)--(3)--(2);
\end{tikzpicture}}}

\newcommand{\flagthree}[4]{%top label, bottom left label, bottom right label, bottomest label}
\resizebox{0.8cm}{0.4cm}{
\begin{tikzpicture}
\node[fill=black, circle, minimum size=0.3cm, label=right:\Huge{\textbf{#1}}] (1) {};
\node[fill=black, circle, minimum size=0.3cm, label=left:\Huge{\textbf{#2}}] (2) [below left of=1] {};
\node[fill=black, circle, minimum size=0.3cm, label=right:\Huge{\textbf{#3}}] (3) [below right of=1] {};
\node[fill=black, circle, minimum size=0.3cm, label=right:\Huge{\textbf{#4}}] (4) [below left of=3] {};
\draw (2)--(1)--(3)--(4);
\end{tikzpicture}}}

\newcommand{\gthree}[4]{%top label, bottom left label, bottom right label, bottomest label}
\resizebox{0.8cm}{0.4cm}{
\begin{tikzpicture}
\node[fill=black, circle, minimum size=0.3cm, label=right:\Huge{\textbf{#1}}] (1) {};
\node[fill=black, circle, minimum size=0.3cm, label=left:\Huge{\textbf{#2}}] (2) [below left of=1] {};
\node[fill=black, circle, minimum size=0.3cm, label=right:\Huge{\textbf{#3}}] (3) [below right of=1] {};
\node[fill=black, circle, minimum size=0.3cm, label=right:\Huge{\textbf{#4}}] (4) [below left of=3] {};
\draw (3)--(2)--(1)--(3)--(4);
\end{tikzpicture}}}

\newcommand{\flagfour}[4]{%top label, bottom left label, bottom right label, bottomest label}
\resizebox{0.8cm}{0.4cm}{
\begin{tikzpicture}
\node[fill=black, circle, minimum size=0.3cm, label=right:\Huge{\textbf{#1}}] (1) {};
\node[fill=black, circle, minimum size=0.3cm, label=left:\Huge{\textbf{#2}}] (2) [below left of=1] {};
\node[fill=black, circle, minimum size=0.3cm, label=right:\Huge{\textbf{#3}}] (3) [below right of=1] {};
\node[fill=black, circle, minimum size=0.3cm, label=right:\Huge{\textbf{#4}}] (4) [below left of=3] {};
\draw (2)--(1)--(3);
\draw (1)--(4)--(2);
\end{tikzpicture}}}

\newcommand{\gfour}[4]{%top label, bottom left label, bottom right label, bottomest label}
\resizebox{0.8cm}{0.4cm}{
\begin{tikzpicture}
\node[fill=black, circle, minimum size=0.3cm, label=right:\Huge{\textbf{#1}}] (1) {};
\node[fill=black, circle, minimum size=0.3cm, label=left:\Huge{\textbf{#2}}] (2) [below left of=1] {};
\node[fill=black, circle, minimum size=0.3cm, label=right:\Huge{\textbf{#3}}] (3) [below right of=1] {};
\node[fill=black, circle, minimum size=0.3cm, label=right:\Huge{\textbf{#4}}] (4) [below left of=3] {};
\draw (3)--(2)--(1)--(3);
\draw (1)--(4)--(2);
\end{tikzpicture}}}

\newcommand{\flagfive}[4]{%top label, bottom left label, bottom right label, bottomest label}
\resizebox{0.8cm}{0.4cm}{
\begin{tikzpicture}
\node[fill=black, circle, minimum size=0.3cm, label=right:\Huge{\textbf{#1}}] (1) {};
\node[fill=black, circle, minimum size=0.3cm, label=left:\Huge{\textbf{#2}}] (2) [below left of=1] {};
\node[fill=black, circle, minimum size=0.3cm, label=right:\Huge{\textbf{#3}}] (3) [below right of=1] {};
\node[fill=black, circle, minimum size=0.3cm, label=right:\Huge{\textbf{#4}}] (4) [below left of=3] {};
\draw (2)--(1)--(3);
\draw (1)--(4)--(3);
\end{tikzpicture}}}

\newcommand{\gfive}[4]{%top label, bottom left label, bottom right label, bottomest label}
\resizebox{0.8cm}{0.4cm}{
\begin{tikzpicture}
\node[fill=black, circle, minimum size=0.3cm, label=right:\Huge{\textbf{#1}}] (1) {};
\node[fill=black, circle, minimum size=0.3cm, label=left:\Huge{\textbf{#2}}] (2) [below left of=1] {};
\node[fill=black, circle, minimum size=0.3cm, label=right:\Huge{\textbf{#3}}] (3) [below right of=1] {};
\node[fill=black, circle, minimum size=0.3cm, label=right:\Huge{\textbf{#4}}] (4) [below left of=3] {};
\draw (3)--(2)--(1)--(3);
\draw (1)--(4)--(3);
\end{tikzpicture}}}

\newcommand{\flagsix}[4]{%top label, bottom left label, bottom right label, bottomest label}
\resizebox{0.8cm}{0.4cm}{
\begin{tikzpicture}
\node[fill=black, circle, minimum size=0.3cm, label=right:\Huge{\textbf{#1}}] (1) {};
\node[fill=black, circle, minimum size=0.3cm, label=left:\Huge{\textbf{#2}}] (2) [below left of=1] {};
\node[fill=black, circle, minimum size=0.3cm, label=right:\Huge{\textbf{#3}}] (3) [below right of=1] {};
\node[fill=black, circle, minimum size=0.3cm, label=right:\Huge{\textbf{#4}}] (4) [below left of=3] {};
\draw (2)--(1)--(3);
\draw (2)--(4)--(3);
\end{tikzpicture}}}

\newcommand{\gsix}[4]{%top label, bottom left label, bottom right label, bottomest label}
\resizebox{0.8cm}{0.4cm}{
\begin{tikzpicture}
\node[fill=black, circle, minimum size=0.3cm, label=right:\Huge{\textbf{#1}}] (1) {};
\node[fill=black, circle, minimum size=0.3cm, label=left:\Huge{\textbf{#2}}] (2) [below left of=1] {};
\node[fill=black, circle, minimum size=0.3cm, label=right:\Huge{\textbf{#3}}] (3) [below right of=1] {};
\node[fill=black, circle, minimum size=0.3cm, label=right:\Huge{\textbf{#4}}] (4) [below left of=3] {};
\draw (3)--(2)--(1)--(3);
\draw (2)--(4)--(3);
\end{tikzpicture}}}

\newcommand{\flagseven}[4]{%top label, bottom left label, bottom right label, bottomest label}
\resizebox{0.8cm}{0.4cm}{
\begin{tikzpicture}
\node[fill=black, circle, minimum size=0.3cm, label=right:\Huge{\textbf{#1}}] (1) {};
\node[fill=black, circle, minimum size=0.3cm, label=left:\Huge{\textbf{#2}}] (2) [below left of=1] {};
\node[fill=black, circle, minimum size=0.3cm, label=right:\Huge{\textbf{#3}}] (3) [below right of=1] {};
\node[fill=black, circle, minimum size=0.3cm, label=right:\Huge{\textbf{#4}}] (4) [below left of=3] {};
\draw (4)--(2)--(1)--(3)--(4)--(1);
\end{tikzpicture}}}

\newcommand{\gseven}[4]{%top label, bottom left label, bottom right label, bottomest label}
\resizebox{0.8cm}{0.4cm}{
\begin{tikzpicture}
\node[fill=black, circle, minimum size=0.3cm, label=right:\Huge{\textbf{#1}}] (1) {};
\node[fill=black, circle, minimum size=0.3cm, label=left:\Huge{\textbf{#2}}] (2) [below left of=1] {};
\node[fill=black, circle, minimum size=0.3cm, label=right:\Huge{\textbf{#3}}] (3) [below right of=1] {};
\node[fill=black, circle, minimum size=0.3cm, label=right:\Huge{\textbf{#4}}] (4) [below left of=3] {};
\draw (4)--(2)--(1)--(3)--(4)--(1);
\draw (2)--(3);
\end{tikzpicture}}}

\newcommand{\vedge}[0]{
\resizebox{0.12cm}{0.35cm}{
\begin{tikzpicture}
\node[fill=black, circle, minimum size=0.3cm] (1) {};
\node[fill=black, circle, minimum size=0.3cm] (2) [below of=1] {};
\draw (2)--(1);
\end{tikzpicture}}}

\newcommand{\vnonedge}[0]{
\resizebox{0.12cm}{0.35cm}{
\begin{tikzpicture}
\node[fill=black, circle, minimum size=0.3cm] (1) {};
\node[fill=black, circle, minimum size=0.3cm] (2) [below of=1] {};
\end{tikzpicture}}}

\newcommand{\labeledvedge}[2]{%top label, bottom label
\resizebox{0.32cm}{0.4cm}{
\begin{tikzpicture}
\node[fill=black, circle, minimum size=0.3cm, label=right:\Huge{\textbf{#1}}] (1) {};
\node[fill=black, circle, minimum size=0.3cm, label=right:\Huge{\textbf{#2}}] (2) [below of=1] {};
\draw (2)--(1);
\end{tikzpicture}}}

\newcommand{\labeledvnonedge}[2]{%top label, bottom label
\resizebox{0.32cm}{0.4cm}{
\begin{tikzpicture}
\node[fill=black, circle, minimum size=0.3cm, label=right:\Huge{\textbf{#1}}] (1) {};
\node[fill=black, circle, minimum size=0.3cm, label=right:\Huge{\textbf{#2}}] (2) [below of=1] {};
\end{tikzpicture}}}

\newcommand{\Fone}[1]{%top label}
\resizebox{0.8cm}{0.4cm}{
\begin{tikzpicture}
\node[fill=black, circle, minimum size=0.3cm, label=right:\Huge{\textbf{#1}}] (1) {};
\node[fill=black, circle, minimum size=0.3cm] (2) [below left of=1] {};
\node[fill=black, circle, minimum size=0.3cm] (3) [below right of=1] {};
\draw (2)--(1);
\end{tikzpicture}}}

\newcommand{\Fones}[1]{%top label}
\resizebox{0.8cm}{0.4cm}{
\begin{tikzpicture}
\node[fill=black, circle, minimum size=0.3cm, label=right:\Huge{\textbf{#1}}] (1) {};
\node[fill=black, circle, minimum size=0.3cm] (2) [below left of=1] {};
\node[fill=black, circle, minimum size=0.3cm] (3) [below right of=1] {};
\draw (3)--(1);
\end{tikzpicture}}}

\date{\today}

\title[Symmetric Sums of Squares over $k$-Subset Hypercubes]{Symmetric Sums of Squares over\\ $k$-Subset Hypercubes}
\thanks{Saunderson and Thomas were partially supported by 
the National Science Foundation under the grants CCF-1409836 and DMS-1418728 respectively. }

\author{Annie Raymond}
\address{Department of Mathematics, University of Washington, Box
  354350, Seattle, WA 98195, USA} \email{raymonda@uw.edu}

\author{James Saunderson}
\address{Department of Electrical Engineering, University of Washington, Box 352500, Seattle, WA 98195, USA} \email{jamesfs@uw.edu}

\author{Mohit Singh}
\address{Microsoft Research,
99/2934, 1 Microsoft Way,
Redmond, WA 98052}
\email{mohits@microsoft.com}

\author{Rekha R. Thomas}
\address{Department of Mathematics, University of Washington, Box
  354350, Seattle, WA 98195, USA} \email{rrthomas@uw.edu}

\begin{document}

\begin{abstract}
We consider the problem of finding
sum of squares (sos) expressions to establish
the non-negativity of a symmetric polynomial over a discrete
hypercube whose coordinates are indexed by the $k$-element subsets of $[n]$.
For simplicity, we focus on the case $k=2$, but our results extend naturally to all values of $k \geq 2$.
We develop a variant of the Gatermann-Parrilo symmetry-reduction method
tailored to our setting that allows for several simplifications and a connection to flag algebras.

We show that every symmetric polynomial that has a sos expression
of a fixed degree also has a succinct sos expression whose size depends
only on the degree and not on the number of variables. Our method bypasses much of the technical
difficulties needed to apply the Gatermann-Parrilo method, and offers flexibility in
obtaining succinct sos expressions that are combinatorially meaningful.
As a byproduct of our results, we arrive at a natural representation-theoretic
justification for the concept of flags as introduced by Razborov in his flag algebra calculus.
Furthermore, this connection exposes a family of non-negative polynomials that
cannot be certified with any fixed set of flags, answering a question of Razborov in the context of our finite setting.
\end{abstract}

\maketitle

% !TEX root =  main.tex
\section{Introduction}

Polynomial optimization over discrete hypercubes plays a central role in many areas 
such as combinatorial optimization, decision problems and proof complexity. 
In many situations, it is natural to consider $k$-subset hypercubes by which we mean 
discrete hypercubes whose coordinates are indexed 
by the $k$-element subsets of a ground set $[n]$. 
For instance, a major focus in extremal graph theory is to 
optimize the edge (hyperedge) density in families of graphs (hypergraphs) with 
specified structure which can be cast as optimization problems 
over $k$-subset hypercubes. In this scenario, as in many others, the polynomial to be optimized is often 
symmetric which allows representation-theoretic techniques to dramatically cut down on computations.
In this paper, we consider the general problem of optimizing a symmetric polynomial over a $k$-subset 
hypercube $\mathcal{V}_{n,k} := \{0,1\}^{\binom{n}{k}}$.  We focus on the case $k=2$, and hence use the 
notation $\mathcal{V}_{n}$ in place of $\mathcal{V}_{n,2} = \{0,1\}^{\binom{n}{2}}$ throughout.
As we will explain in Section~\ref{sec:discussion}, our results extend naturally to all $k \geq 2$.

Phrased differently, our central problem is to certify the non-negativity of a symmetric polynomial $\mathsf{p}$ over 
$\mathcal{V}_n$ which can be done by finding a sum of squares (sos) expression that equals $\mathsf{p}$ as a function on $\mathcal{V}_n$. In \cite{GatermannParrilo}, Gatermann and Parrilo showed how to use representation theory to simplify the computations involved in finding a sos representation of a polynomial $\mathsf{p}$ that is invariant under the action of a finite group. 
We propose a variant of their method that is adapted to the 
combinatorics in our setting, and hence, offers many simplifications and advantages. Among the highlights is 
a proof that if $\mathsf{p}$ has a sos certificate of degree $d$, then it also has a succinct sos expression whose 
size depends on $d$ but is independent of $n$. Secondly, we show that Razborov's theory of flags makes a natural 
entry into this problem as a byproduct of the representation theory of the symmetric group. 

We now introduce some notation that will help us elaborate on our results. Fix a positive integer $n$ and let $\mathbb{R}[{\mathsf{x}}]:=\mathbb{R}[\mathsf{x}_{ij} \,:\, 1 \leq i < j \leq n]$ be the polynomial ring in variables indexed by the pairs $ij$ where $1 \leq i < j \leq n$.  The set of polynomials that vanish on the discrete hypercube $\mathcal{V}_n$ is the ideal $\mathcal{I}_n$ in $\mathbb{R}[{\mathsf{x}}]$ generated by the ${n \choose 2}$ polynomials $\mathsf{x}_{ij}^2 - \mathsf{x}_{ij}$ for all $1 \leq i < j \leq n$. The set of functions on $\mathcal{V}_n$ may be identified with the quotient ring $\mathbb{R}[\mathsf{x}] / \mathcal{I}_n$. 
Recall that two polynomials $f,g \in \mathbb{R}[\mathsf{x}]$ represent the same function in $\mathbb{R}[\mathcal{V}_n]$ 
if and only if $f - g \in \mathcal{I}_n$, written as $f \equiv g \textup{ mod } \mathcal{I}_n$. This algebraic language is helpful in examples. 
The elements of $\mathbb{R}[\mathsf{x}] / \mathcal{I}_n$ are in bijection with the {\em square-free} polynomials in $\mathbb{R}[\mathsf{x}]$, namely, those polynomials in which every monomial is square-free or multilinear.
As a vector space, it will be convenient to identify $\mathbb{R}[\mathcal{V}_n]$ with the set of all square-free polynomials in $\mathbb{R}[\mathsf{x}]$. Let $\mathbb{R}[{\mathsf{x}}]_{\leq d}$ denote the polynomials of degree at most $d$ in $\mathbb{R}[{\mathsf{x}}]$. Denote by $\mathbb{R}[\mathcal{V}_n]_{\leq d}$ the set of functions of degree at most $d$ in $\mathbb{R}[\mathcal{V}_n]$, namely the quotient ring 
$\mathbb{R}[\mathsf{x}]_{\leq d}/\mathcal{I}_n$.  As a vector space, $\mathbb{R}[\mathcal{V}_n]_{\leq d}$ can be identified with the set of all square-free polynomials of degree at most $d$.

The symmetric group $\mathfrak{S}_n$ acts on $\mathbb{R}[{\mathsf{x}}]$ by linearly extending the action of $\mathfrak{S}_n$ on monomials via $\mathfrak{s} \cdot \mathsf{x}_{ij}:=\mathsf{x}_{\mathfrak{s}(i)\mathfrak{s}(j)}$ for each $\mathfrak{s}\in \mathfrak{S}_n$. Since the ideal $\mathcal{I}_n$ is invariant under this action, the action of $\mathfrak{S}_n$ extends to $\mathbb{R}[\mathcal{V}_{n}]$ and $\mathbb{R}[\mathcal{V}_n]_{\leq d}$. We say that a polynomial $\mathsf{p}$ is {\em symmetric} or $\mathfrak{S}_n$-{\em invariant} if it is fixed under this action of $\mathfrak{S}_n$. 

Every non-negative $\mathsf{p}$ in $\mathbb{R}[\mathcal{V}_n]$ is a sum of squares, i.e., $\mathsf{p} = \sum \mathsf{f}_i^2$ 
for finitely many $\mathsf{f}_i \in \mathbb{R}[\mathcal{V}_n]$ \cite{PabloTechReport}, see also \cite[Theorem 2.4]{MoniqueSurvey}. If the $\mathsf{f}_i$ are restricted to come from a fixed subspace $V\subseteq \mathbb{R}[\mathcal{V}_n]$, then we say that $\mathsf{p}$ is $V$-sos.  If $p$ is $\mathbb{R}[\mathcal{V}_{n}]_{\leq d}$-sos, we simply say that $\mathsf{p}$ is $d$-sos.

If $\mathsf{p} \in \mathbb{R}[\mathcal{V}_n]$ is $V$-sos, then a sos certificate for it can be found by solving a semidefinite program (SDP) 
 (see, for instance,~\cite[Chapter 3]{BPTSIAMBook}). 
Indeed, this can be done by finding a positive semidefinite (psd) matrix
$Q$ such that 
$\mathsf{p} = \textup{tr}( Q\, \bm{\mathsf{v}}\bm{\mathsf{v}}^\top )$, 
where $\bm{\mathsf{v}}$ is a vector containing the elements in a basis of $V$, and $\textup{tr}(\cdot)$ stands for trace. In particular, if $\mathsf{p}$ is $d$-sos, we could choose the entries of $\bm{\mathsf{v}}$ 
to consist of all square-free monomials of degree at most $d$. Since
the number of such monomials is $\displaystyle{\sum_{i=0}^{d}\binom{\binom{n}{2}}{i}}$, the size of $Q$
becomes unwieldy as $n$ grows.

If $\mathsf{p}$ is $d$-sos and invariant with respect to a finite group, Gatermann and Parrilo~\cite{GatermannParrilo}
show how to find sos expressions for $\mathsf{p}$ by solving a (potentially) much smaller SDP than the one mentioned above.
In our setting, their starting point is the isotypic decomposition of the vector
space $V = \RR[\cV_n]_{\leq d}$ under the action of $\fS_n$. This 
decomposition has the form $V = \bigoplus_{\bm{\lambda}}  V_{\bm{\lambda}}$ 
where ${\bm{\lambda}}$ is a partition of $n$, and each isotypic $V_{\bm{\lambda}}$ is in turn a direct
sum of $m_{\bm{\lambda}}$ copies of a single irreducible representation $S_{\bm{\lambda}}$ for $\fS_n$.
Applying the results of~\cite{GatermannParrilo} to our setting, one can show that  
there exist $m_{\bm{\lambda}}\times m_{\bm{\lambda}}$ psd matrices $\tilde{Q}_{\bm{\lambda}}$ such that 
\begin{equation}
\label{eq:GP-intro}
\mathsf{p} = \sum_{\bm{\lambda}} 
\textup{tr}\left( \tilde{Q}_{\bm{\lambda}}\, \sym(\bm{\tilde{\mathsf{v}}_{\bm{\lambda}}}\bm{\tilde{\mathsf{v}}_{\bm{\lambda}}}^\top)\right).
\end{equation}
Here, for each partition $\bm{\lambda}$, the entries of $\bm{\tilde{\mathsf{v}}_{\bm{\lambda}}}$ consist of elements in a 
basis for its $m_{\bm{\lambda}}$-dimensional multiplicity space 
(which is discussed in Section~\ref{sec:bases_sn} and Appendix~\ref{app:appendix1}). The operation $\textup{sym}(\cdot)$ 
symmetrizes its argument under the group action.
Hence, if $\mathsf{p}$ is symmetric and $d$-sos, 
we can find a symmetry-reduced sos description of $\mathsf{p}$ by solving a SDP
of size $\sum_{\bm{\lambda}} m_{\bm{\lambda}}$.

A priori, it is unclear how much smaller  the size of the Gatermann-Parrilo SDP is compared to the dimension of $V$.
Another major challenge is that constructing this SDP requires the explicit knowledge of a 
basis for each of the multiplicity spaces (i.e., appropriate vectors $\bm{\tilde{\mathsf{v}}}_{\bm{\lambda}}$). 
The algorithm outlined in~\cite{GatermannParrilo} requires constructing a symmetry-adapted basis for $V$ (see, e.g.,~\cite{FaesslerStiefelBook}), 
which is computationally costly, having complexity that depends on both $n$ and $d$.

In this paper, we focus on finding sos certificates for symmetric polynomials over discrete hypercubes $\mathcal{V}_{n,k}$ when 
$k\geq 2$. In these cases, little is explicitly known about how the 
space $\RR[\cV_{n,k}]$ of square-free polynomials in $\binom{n}{k}$ variables decomposes under the action of $\fS_n$ 
on $k$-tuples. This makes a direct application of the Gatermann-Parrilo method infeasible in this setting. This is in sharp contrast with the 
simpler case $k=1$ for which 
the isotypic decomposition of $\RR[\cV_{n,1}]_{\leq d}$, and hence all the information required for the Gatermann-Parrilo SDP, 
is explicitly known (see, e.g.,~\cite{BlekhermanGouveiaPfeiffer}). This decomposition underpins a number of recent results related to non-negative 
$\fS_n$-invariant polynomials on $\{0,1\}^n$~\cite{BlekhermanGouveiaPfeiffer,lee2016sum,kurpisz2016sum}.

\subsection{Our results} 
\label{sec:our-results}
Our first contribution is to show that if $\mathsf{p}$ is symmetric and $d$-sos, 
then it has a symmetry-reduced sos certificate that can be obtained by solving a SDP of size independent of $n$.
Throughout, we will informally refer to such a sos expression as being \emph{succinct}.
We establish this result in two steps. First,
we show that the number of partitions needed in~\eqref{eq:GP-intro} is bounded above by
$p(0)+p(1)+p(2)+\cdots +p(2d)$ (where $p(i)$ is the number of partitions of $i$), a quantity 
that is independent of $n$. Second, we show that each $m_{\bm{\lambda}}$ is also bounded above by a quantity that is 
independent of $n$ (see Proposition~\ref{prop:mlambda-indn} for a precise statement).

Our next contribution is to develop a variant of the Gatermann-Parrilo method tailored to our setting,
that circumvents the representation-theoretic difficulties inherent in constructing bases for the multiplicity spaces. 
In particular, we show how to construct subspaces that contain the multiplicity spaces, have dimension 
depending only on $d$, and that are spanned by combinatorially meaningful polynomials that can be enumerated easily.
One such spanning set arises as polynomial analogs 
of a specific collection of partially labeled graphs present in both the theory of graph homomorphisms \cite{LovaszBook} and flag algebras \cite{RazborovFlagAlgebras}. Partially labeled graphs are called \emph{flags} in the latter. We will use this terminology for brevity and since we make several connections 
to the work in \cite{RazborovFlagAlgebras}.  

A key result coming from the connections we make to flags is that if $\mathsf{p}$ is symmetric and $d$-sos,
then there is an sos expression for $\mathsf{p}$ that uses only flag polynomials from flags on at most $2d$ vertices.
Further, we prove that even particular restricted flag sos expressions commonly
found in the flag algebra literature suffice for this result. Together with
\cite{RST}, this implies the surprising result that flag methods are just as
capable as standard symmetry-reduction methods for providing sos certificates
for symmetric polynomials over $\mathcal{V}_n$. Moreover, this connection to
flags offers a family of symmetric polynomials of fixed degree that cannot be
certified with any fixed set of flags, answering a question in
\cite{RazborovFlagAlgebras} in the finite setting.

The theories of flag algebras and graph homomorphisms have emerged as powerful
tools for establishing  asymptotic (symmetric) inequalities among graph
densities in extremal graph theory, by expressing them as sums of squares up to
an error that vanishes asymptotically.  Viewed as optimization problems, these are instances of polynomial
optimization with infinitely many variables. 
Our results show that there is much to be
gained by pausing at a finite $n$.  We show that flags can be used to exactly certify
the non-negativity of {\em any} symmetric polynomial over the discrete
hypercube $\mathcal{V}_{n}$ and its subsets. We illustrate our methods through a problem from Ramsey theory
and a Tur\'an problem for $4$-cycles.

\subsection{Organization of the paper}
This paper is organized as follows. In Section~\ref{sec:prelim}, we introduce some concepts from the representation theory of $\fS_n$.
In particular, we define a key ingredient for the paper, namely subspaces $W_{\tau_{\bm{\lambda}}}$ indexed by tableaux of shape $\bm{\lambda}$,
that  are isomorphic to the multiplicity spaces referred to in Section~\ref{sec:our-results}. These subspaces play a crucial role in the 
Gatermann-Parrilo method tailored to our setting, which we present in Section~\ref{sec:GP}. A proof can be found in Appendix~\ref{app:appendix1}.
We prove in Section~\ref{sec:iso-support} that if $d$ is fixed, then 
the number of partitions that are needed in the resulting sos certificate
for $\mathsf{p}$ is bounded by a function of $d$ that is independent of $n$. 
In Section~\ref{sec:GP via spanning sets}, we provide a variant of the Gatermann-Parrilo method that only requires spanning sets for the 
subspaces $W_{\tau_{\bm{\lambda}}}$ as opposed to bases. 
In Section~\ref{sec:spanning sets}, we construct three such interrelated spanning sets,
all of which have the property that their sizes depend on $d$ but not $n$. In
the process, we naturally encounter flags and their densities as studied by Razborov. 
In Section~\ref{sec:razborov sos}, we prove that the flag sos expressions that emerge from our method
can be simplified to take on the more restricted form commonly found in the flag algebra literature. 
We also show that every symmetric $d$-sos expression can be retrieved from flags on at most $2d$ vertices. 
Our methods naturally extend to subsets of $\mathcal{V}_n$. 
We illustrate this in Section~\ref{sec:examples} by giving sos proofs of two known results. 
The first is an upper bound on the edge density of a $n$-vertex graph that avoids $4$-cycles, 
and the second is that the Ramsey
number $R(3,3)=6$.  The technical details of the Ramsey proof are in 
Appendix~\ref{app:Ramsey sos}.
In Section~\ref{sec:infinite_example}, we give a family of non-negative symmetric polynomials
that cannot be certified by a fixed collection of flags.
We close the paper in Section~\ref{sec:discussion} with a discussion of our results, and their extension to hypergraphs.

\smallskip

\noindent{\bf{Acknowledgments.}} Several people offered valuable input to this paper. 
In particular, we thank Albert Atserias, Monty McGovern, Pablo Parrilo and Sasha Razborov. We 
especially thank Greg Blekherman for his help with several facets of the representation theory 
that underlies this work.

% !TEX root =  main.tex
\section{Sums of squares certificates for symmetric polynomials}
\label{sec:bases_sn}

Suppose that $\mathsf{p}\in \RR[\cV_n]$ is $d$-sos. In other words, suppose
that $\mathsf{p} = \sum \mathsf{f}_i^2$
where $\mathsf{f}_i\in \RR[\cV_n]_{\leq d}$ for all $i$. 
If, in addition, $\mathsf{p}$ is fixed by the action of $\fS_n$
on $\RR[\cV_n]$, Gatermann and Parrilo~\cite{GatermannParrilo} showed how to find simpler, symmetry-reduced, 
sos expressions for $\mathsf{p}$ by exploiting the isotypic decomposition of $\RR[\cV_n]_{\leq d}$ 
under the action of $\fS_n$. In this section, we first introduce preliminary definitions and
describe the structure of these symmetry-reduced sos expressions of Gatermann and Parrilo (Theorem~\ref{thm:GP}). 
To search for these sos expressions via semidefinite programming requires knowledge of bases 
for certain subspaces $W_{\tau_{\bm{\lambda}}}$ of $\RR[\cV_n]_{\leq d}$ (defined in Section~\ref{sec:prelim}). 
In Section~\ref{sec:iso-support}, we show that the number of partitions $\bm{\lambda}$ of $n$ for which these subspaces 
are non-zero is bounded by a quantity that depends on $d$ but is independent of $n$ (Corollary~\ref{cor:isodecomp}). 
In Section~\ref{sec:GP via spanning sets}, we introduce the more flexible variant of Theorem~\ref{thm:GP} that does not 
require explicit knowledge of bases for the $W_{\tau_{\bm{\lambda}}}$, but instead allows us to work with 
spanning sets for these spaces.

\subsection{Preliminaries}
\label{sec:prelim}
Recall that the symmetric group $\fS_n$ acts on $\RR[{\mathsf{x}}]_{\leq d}$, the ring of polynomials
of degree at most $d$ in $\binom{n}{2}$ variables, by sending each
variable $\mathsf{x}_{ij}$ to $\fs\cdot \mathsf{x}_{ij} =
\mathsf{x}_{\fs(i)\fs(j)}$ for each $\fs\in \fS_{n}$. This descends to
an action on $\RR[\cV_{n}]_{\leq d} = \RR[{\mathsf{x}}]_{\leq d}/\mathcal{I}_n$ because
$\mathcal{I}_n$ is invariant under the action of $\fS_n$. By extending these actions 
linearly, we can regard $\RR[{\mathsf{x}}]_{\leq d}$ and $\RR[\cV_n]_{\leq d}$ as $\fS_n$-modules. 

We briefly summarize standard facts and terminology related to the representation theory of $\fS_n$ (see, e.g.,~\cite[Chapters 1\& 2]{SaganBook}).
A $\fS_n$-module $U$ is \emph{irreducible} if the only subspaces of $U$ that are invariant under the 
action of $\fS_n$ are $\{0\}$ and $U$. The irreducible $\fS_n$-modules are
indexed by \emph{partitions} of $n$, namely sequences of positive integers
${\bm{\lambda}} = (\lambda_1,\ldots,\lambda_k)$ such that $\lambda_1\geq \cdots
\geq \lambda_k > 0$ and $\lambda_1+\cdots+\lambda_k = n$.  Each $\lambda_i$ is
called a \emph{part}. We use the shorthand $\bm{\lambda} \vdash n$ to indicate
that $\bm{\lambda}$ is a partition of $n$.  We denote the irreducible
$\fS_n$-module indexed by the partition ${\bm{\lambda}} \vdash n$ by
$S_{\bm{\lambda}}$.  Its dimension is denoted by $n_{\bm{\lambda}}$.

Any $\fS_n$-module $V$ has an \emph{isotypic decomposition}
\begin{equation}
	\label{eq:isotypic} 
	V = \bigoplus_{{\bm{\lambda}} \vdash n} V_{\bm{\lambda}}, 
\end{equation}
which expresses $V$ as a direct sum of $\fS_n$-modules $V_{\bm{\lambda}}$, 
called the \emph{isotypic components}. Each $V_{\bm{\lambda}}$ is the span of 
all possible isomorphic copies of the irreducible $\fS_n$-module $S_{\bm{\lambda}}$
in $V$. (We define these notions precisely in Appendix~\ref{app:appendix1}.)

Combinatorial objects, called tableaux, and related subgroups of $\fS_n$, called row groups, play an important role
in the representation theory of $\fS_n$, and appear throughout this paper. 
A partition ${\bm{\lambda}}\vdash n$ has a \emph{shape} (or \emph{Young diagram}) with rows of size $\lambda_1 \geq \cdots \geq \lambda_k$. 
A \emph{tableau} of shape ${\bm{\lambda}}$, denoted $\tau_{\bm{\lambda}}$, is a filling of the $n$ boxes in the diagram of ${\bm{\lambda}}$ with the numbers 
$1,2,\ldots,n$ bijectively. A tableau is \emph{standard} if the numbering increases from left to right along each row and down each column. 
The  number of standard tableaux of shape ${\bm{\lambda}}$ is $n_{\bm{\lambda}}$, the dimension of the 
irreducible $\fS_n$-module $S_{\bm{\lambda}}$~\cite[Theorem~2.5.2]{SaganBook}.  

If $\tau_{\bm{\lambda}}$ is a tableau of shape ${\bm{\lambda}}$, let
$\textup{row}_i(\tau_{\bm{\lambda}})$ denote the set of labels  in the $i$th row of
$\tau_{\bm{\lambda}}$. The \emph{row group} $\fR_{\tau_{\bm{\lambda}}}$ of the tableau
$\tau_{\bm{\lambda}}$ is the subgroup of $\fS_n$ that leaves each
$\textup{row}_i(\tau_{\bm{\lambda}})$ invariant, i.e., 
\[ \fR_{\tau_{\bm{\lambda}}} :=\{ \fs\in \fS_n: \fs \cdot \textup{row}_i(\tau_{\bm{\lambda}})
= \textup{row}_i(\tau_{\bm{\lambda}})\;\;\text{for all $i=1,2,\ldots,k$}\}.\]

If $U$ is a $\fS_n$-module and $\tau_{\bm{\lambda}}$ is a tableau, let $U^{\fR_{\tau_{\bm{\lambda}}}}$ denote the 
subspace of $U$ consisting of points fixed by the row group $\fR_{\tau_{\bm{\lambda}}}$, i.e.,
\[ U^{\fR_{\tau_{{\bm{\lambda}}}}} := \{u\in U:\; \fs \cdot u = u\;\;\text{for all $\fs\in \fR_{\tau_{\bm{\lambda}}}$}\}.\]
The following definition plays a central role in the paper. 
\begin{definition}
	\label{def:W-tau-lambda}
	If $\RR[\cV_n]_{\leq d} = \bigoplus_{{\bm{\lambda}}} V_{\bm{\lambda}}$ is the isotypic decomposition of $\RR[\cV_n]_{\leq d}$, and $\tau_{\bm{\lambda}}$ is a 
	tableau of shape ${\bm{\lambda}}$, define
	\[ W_{\tau_{\bm{\lambda}}} := V_{\bm{\lambda}}^{\fR_{\tau_{\bm{\lambda}}}}\]
	to be the subspace of the isotypic $V_{\bm{\lambda}}$ fixed by the action of the row group $\fR_{\tau_{\bm{\lambda}}}$. 
\end{definition}
Since $V_{\bm{\lambda}}$ is a subspace of $\RR[\cV_n]_{\leq d}$, it 
follows that $W_{\tau_{\bm{\lambda}}}$ is a subspace of $\RR[\cV_n]_{\leq d}^{\fR_{\tau_{\bm{\lambda}}}}$, the subspace of $\RR[\cV_n]_{\leq d}$ that is fixed by 
the action of the row group $\fR_{\tau_{\bm{\lambda}}}$. This simple observation is the main property of
$W_{\tau_{\bm{\lambda}}}$ that we use in subsequent sections of the paper. 

In Lemma~\ref{lem:mult-iso-2} of Appendix~\ref{app:appendix1}, we show that for any choice of tableau $\tau_{\bm{\lambda}}$ of shape $\bm{\lambda}$, there is 
a vector space isomorphism between $W_{\tau_{\bm{\lambda}}}$ and the \emph{multiplicity space} of the irreducible $\fS_n$-module
$S_{\bm{\lambda}}$ in $\RR[\cV_n]_{\leq d}$ (defined in Appendix~\ref{app:appendix1}). As such, $\dim(W_{\tau_{\bm{\lambda}}})$ is the same for 
any tableau $\tau_{\bm{\lambda}}$ of shape $\bm{\lambda}$. We denote this  dimension by $m_{\bm{\lambda}}$, and often 
refer to $m_{\bm{\lambda}}$ as the \emph{multiplicity} of $S_{\bm{\lambda}}$ in $V$. Some of the multiplicities 
$m_{\bm{\lambda}}$ may be zero as we will see in Section~\ref{sec:iso-support}.

\subsection{Symmetry-reduced sos expressions of Gatermann and Parrilo}
\label{sec:GP}
We are now in a position to state a version of the main result of~\cite{GatermannParrilo} tailored to our setting. 
This result tells us the structure of symmetry-reduced sos expressions for $\fS_n$-invariant $d$-sos polynomials, and shows 
that we can search for such sos expressions by solving a SDP of size $\sum_{\bm{\lambda} \vdash n} m_{\bm{\lambda}}$. 

\begin{theorem} \cite{GatermannParrilo} \label{thm:GP}
Suppose $\mathsf{p} \in \mathbb{R}[\cV_n]$ is $\mathfrak{S}_n$-invariant and $d$-sos.
For each partition $\bm{\lambda} \vdash n$, fix a tableau $\tau_{\bm{\lambda}}$ of shape $\bm{\lambda}$ and choose a vector space basis 
$\{ \mathsf{b}_1^{\tau_{\bm{\lambda}}}, \ldots, \mathsf{b}_{m_{\bm{\lambda}}}^{\tau_{\bm{\lambda}}} \}$ for $W_{\tau_{\bm{\lambda}}}$. 
Then for each partition $\bm{\lambda}$ of $n$, there exists a $m_{\bm{\lambda}} \times m_{\bm{\lambda}}$ psd matrix $Q_{\bm{\lambda}}$ such that 
\begin{equation}
	\label{eq:GP} \mathsf{p} = \sum_{\bm{\lambda} \vdash n} \textup{tr}( Q_{\bm{\lambda}} \, Y^{\tau_{\bm{\lambda}}}),
\end{equation}
where $Y^{\tau_{\bm{\lambda}}}_{ij} := \textup{sym}(\mathsf{b}_i^{\tau_{\bm{\lambda}}} \mathsf{b}_j^{\tau_{\bm{\lambda}}})$.
\end{theorem}
We provide a full proof of Theorem~\ref{thm:GP} and the background needed in Appendix~\ref{app:appendix1}. Here, we 
make various comments about the statement of this theorem. 

Notice that the matrix $Y^{\tau_{\bm{\lambda}}}$ in the statement of Theorem~\ref{thm:GP} is filled with 
polynomials in ${\mathsf{x}}$, and that it is psd for all evaluations of ${\mathsf{x}}$ on $\cV_n$. Hence 
the right-hand side of~\eqref{eq:GP} is clearly non-negative as a function on $\cV_n$. The content of Theorem~\ref{thm:GP}
is that all $d$-sos $\fS_n$-invariant polynomials have certificates of non-negativity of this form. 
Our statement of Theorem~\ref{thm:GP} does not require a full symmetry-adapted basis (see~\cite{FaesslerStiefelBook}) for $V = \RR[\cV_n]_{\leq d}$. 
Instead it requires, for each partition $\bm{\lambda}$, a basis for $W_{\tau_{\bm{\lambda}}}$ 
for a single fixed tableau of shape $\bm{\lambda}$. That is why we have a SDP of size $\sum_{\bm{\lambda} \vdash n} m_{\bm{\lambda}}$
rather than $\dim(V) = \sum_{\bm{\lambda}\vdash n} m_{\bm{\lambda}} n_{\bm{\lambda}}$. 

We also note that the original Gatermann-Parrilo method is about sos certificates for globally non-negative invariant polynomials, 
but it is easily adapted to the situation of non-negativity over an algebraic variety. In addition,~\cite{GatermannParrilo} 
gives a more refined symmetry-reduced sos result in terms of invariant theory.

In the next two subsections, we will see two important improvements to the statement of Theorem~\ref{thm:GP} that will  eventually 
establish one of our main results that, whenever $\mathsf{p}$ is symmetric and $d$-sos, then $\mathsf{p}$ has a succinct 
sos expression whose size is independent of $n$.

\subsection{Partitions needed in the Gatermann-Parrilo sos}
\label{sec:iso-support}
Recall that the vector space $W_{\tau_{\bm{\lambda}}}$ has dimension $m_{\bm{\lambda}}$ and contributes the sos 
expression $\textup{tr}(Q_{\bm{\lambda}} \, Y^{\tau_{\bm{\lambda}}})$ to the sos certificate of $\mathsf{p}$ in \eqref{eq:GP} where 
$Q_{\bm{\lambda}}$ has size $m_{\bm{\lambda}} \times m_{\bm{\lambda}}$. We now investigate which multiplicities $m_{\bm{\lambda}}$ are non-zero, or equivalently, which partitions are needed 
in the Gatermann-Parrilo symmetry-reduced sos expression. 
We show that the number of partitions that are needed in \eqref{eq:GP} depends on $d$ but is 
independent of $n$. 

We begin by 
investigating which partitions ${\bm{\lambda}} \vdash n$ 
can appear in the isotypic decomposition of $\RR[{\mathsf{x}}]_{\leq d}$.
Our main tool is the following consequence of Young's rule, established in~\cite[Theorem 4.9]{RST}. In what follows, we write
$\bm{\lambda}\geq _{\textup{lex}} \bm{\mu}$ for $\bm{\lambda},\bm{\mu} \vdash n$ if $\bm{\lambda}$ is lexicographically greater than or equal to $\bm{\mu}$, and $\bm{\lambda} \unrhd \bm{\mu}$
if $\bm{\lambda}\geq _{\textup{lex}}\bm{\mu}$ and, in addition, $\bm{\lambda}$ has at most as many parts as $\bm{\mu}$. 

\begin{lemma}
	\label{lemma49}
	If $V$ is a finite-dimensional $\fS_n$-module and $\tau_{\bm{\mu}}$ is a tableau of shape ${\bm{\mu}}$, then 
	\[ V^{\fR_{\tau_{\bm{\mu}}}}\subseteq \bigoplus_{{\bm{\lambda}} \unrhd {\bm{\mu}}} V_{\bm{\lambda}}.\]
\end{lemma}
Recall that a {\em hook partition} of $n$ is one in which all parts except the first has size one, which we denote 
as $(\lambda_1, 1^{n-\lambda_1})$ where $\lambda_1$ is the size of the first part, and where the remaining $n-\lambda_1$ parts 
constitute the \emph{tail} of the hook.
We now use Lemma~\ref{lemma49} to show that only partitions lexicographically at 
least as large as the hook partition $(n-2d,1^{2d})$ can appear in the isotypic decomposition of $V=\RR[{\mathsf{x}}]_{\leq d}$. 
In Proposition~\ref{prop:total number of m lambdas}, we will see that the number of such partitions is bounded by a function of $d$ that is independent of $n$.
\begin{theorem}
	\label{thm:isodecomp}
	The multiplicity $m_{\bm{\lambda}}$ of $S_{\bm{\lambda}}$ in the decomposition of 
	$V=\mathbb{R}[{\mathsf{x}}]_{\leq d}$ into irreducible $\fS_n$-modules is zero unless 
	${\bm{\lambda}}\geq_{\textup{lex}}(n-2d,1^{2d})$, i.e.,  
	 \[ V = \bigoplus_{{\bm{\lambda}} \geq_{\textup{lex}} (n-2d,1^{2d})} V_{\bm{\lambda}}.\]  
\end{theorem}
\begin{proof}
	Let $\mathsf{x}_{i_1j_1}^{\alpha_1}\mathsf{x}_{i_2j_2}^{\alpha_2}\cdots \mathsf{x}_{i_kj_k}^{\alpha_k}$ be a monomial of degree $\sum_{i=1}^{k}\alpha_i\leq d$. 
	The set $L = \{i_1,j_1,i_2,j_2,\ldots,i_k,j_k\}\subseteq [n]$ of indices appearing in the monomial
	has size at most $2d$. Let $\tau$ be a tableau of shape $(n-2d,1^{2d})$
	with the elements of $L$ put into the tail and the remaining $n - |L|$ labels filled arbitrarily into the rest of the Young diagram of $\tau$. Since $\fR_{\tau}$ fixes every element of $L$, clearly 
	$\mathsf{x}_{i_1j_1}^{\alpha_1}\mathsf{x}_{i_2j_2}^{\alpha_2}\cdots \mathsf{x}_{i_kj_k}^{\alpha_k}\in V^{\fR_\tau}$.
	Since $V$ is spanned by monomials of degree at most $d$, we have that 
	\[ V \subseteq \sum_{\textup{shape}(\tau)=(n-2d,1^{2d})} V^{\fR_{\tau}} \subseteq \bigoplus_{{\bm{\lambda}} \unrhd (n-2d,1^{2d})} V_{\bm{\lambda}},\]
	where the second inclusion follows from Lemma~\ref{lemma49}.
	 To finish the proof, we need only show that 
	for the hook shape $(n-2d,1^{2d})$, we have ${\bm{\lambda}} \unrhd (n-2d,1^{2d})$ if and only if ${\bm{\lambda}} \geq_{\textup{lex}} (n-2d,1^{2d})$. 
	It is enough to argue that if ${\bm{\lambda}} \geq_{\textup{lex}} (n-2d,1^{2d})$, then ${\bm{\lambda}}$ has at most as many parts as $(n-2d,1^{2d})$, 
	i.e.,\ ${\bm{\lambda}}$ has at most $2d+1$ parts. This is clearly the case since ${\bm{\lambda}} \vdash n$ and $\lambda_1 \geq n-2d$. 
\end{proof}

By the same argument (but only considering square-free monomials), the same conclusion holds for the isotypic 
decomposition of $V = \RR[\cV_n]_{\leq d}$.

\begin{corollary} \label{cor:isodecomp}
	The multiplicity $m_{\bm{\lambda}}$ of $S_{\bm{\lambda}}$ in the decomposition of 
	$V=\mathbb{R}[\cV_n]_{\leq d}$ into irreducible $\fS_n$-modules or, equivalently, the dimension of $W_{\tau_{\bm{\lambda}}}$ for 
	any tableau of shape $\bm{\lambda}$,  is zero unless 
	${\bm{\lambda}}\geq_{\textup{lex}}(n-2d,1^{2d})$, i.e.,  
	 \[ V = \bigoplus_{{\bm{\lambda}} \geq_{\textup{lex}} (n-2d,1^{2d})} V_{\bm{\lambda}}.\]  
\end{corollary}

For the rest of the body of the paper, we will fix 
\begin{equation*} \label{eq:big-lambda}
\Lambda := \{ \bm{\lambda} \vdash n \,:\, {\bm{\lambda}} \geq_{\textup{lex}} (n-2d,1^{2d})\}.
\end{equation*}
Then, using Theorem~\ref{thm:isodecomp}, we obtain the following corollary to Theorem~\ref{thm:GP}.

\begin{corollary} \label{cor:GP with big lambda}
Suppose $\mathsf{p} \in \mathbb{R}[\cV_n]$ is $\mathfrak{S}_n$-invariant and $d$-sos.
For each partition $\bm{\lambda} \vdash n$, fix a tableau $\tau_{\bm{\lambda}}$ of shape $\bm{\lambda}$ and choose a vector space basis 
$\{ \mathsf{b}_1^{\tau_{\bm{\lambda}}}, \ldots, \mathsf{b}_{m_{\bm{\lambda}}}^{\tau_{\bm{\lambda}}} \}$ for $W_{\tau_{\bm{\lambda}}}$.
Then for each partition $\bm{\lambda} \in \Lambda$,  there exists a $m_{\bm{\lambda}} \times m_{\bm{\lambda}}$ psd matrix $Q_{\bm{\lambda}}$ such that 
$$ \mathsf{p} = \sum_{\bm{\lambda} \in \Lambda} \textup{tr}( Q_{\bm{\lambda}} \, Y^{\tau_{\bm{\lambda}}}).$$
\end{corollary}

We now give an explicit upper bound on the number of partitions that are lexicographically
greater than or equal to the hook $(n-2d,1^{2d})$, and hence a bound on the number of partitions $\bm{\lambda}$
that appear in the isotypic decomposition of $\RR[\cV_n]_{\leq d}$. 
\begin{proposition} \label{prop:total number of m lambdas}
	Let $m_{\bm{\lambda}}$ be the multiplicity of $S_{\bm{\lambda}}$ in the isotypic decomposition of 
	$V = \RR[\cV_n]_{\leq d}$. Then the number of partitions $\bm{\lambda}$ of $n$ such that $m_{\bm{\lambda}}$ is non-zero
	is bounded above by $p(0) + p(1) + p(2)+\cdots + p(2d)$, where $p(i)$ is the number of partitions of $i$. In particular, 
	the size of $\Lambda$ is independent of $n$.
\end{proposition}
\begin{proof}
	By Corollary~\ref{cor:isodecomp}, the number of partitions $\bm{\lambda}$ such that $m_{\bm{\lambda}}$ is 
	non-zero is bounded above by $|\Lambda|$.
	If $\bm{\lambda}\in \Lambda$, then it has the form $(n-2d+i,\bm{\mu})$,
where $\bm{\mu}\vdash 2d-i$ and $i$ is some integer satisfying $0\leq i\leq
2d$. The number of such partitions $\bm{\mu}$ is bounded above by $p(0) +
p(1)+p(2)+\cdots + p(2d)$. 
\end{proof}

By Proposition~\ref{prop:total number of m lambdas}, we see that even though the number of partitions of $n$ grows with $n$, the number of those 
that are needed in the Gatermann-Parrilo symmetry-reduced sos expression for 
a $d$-sos symmetric polynomial $\mathsf{p}$ is bounded by a function of $d$ that is independent of $n$. This is our first improvement
 to the statement of Theorem~\ref{thm:GP} in our setting. 

\subsection{Symmetry-reduction via spanning sets}
\label{sec:GP via spanning sets}

Next, we show that one does not need bases for the $W_{\tau_{\bm{\lambda}}}$ to find a 
sos certificate as in Theorem~\ref{thm:GP}. Instead, it suffices to have a set of polynomials whose span contains $W_{\tau_{\bm{\lambda}}}$. 
This relaxation offers a great deal of flexibility in setting up the SDP in \eqref{eq:GP}.
 
\begin{theorem} \label{thm:GP with spanning set}
Suppose $\mathsf{p}\in \RR[\cV_n]$ is $\mathfrak{S}_n$-invariant and $d$-sos.
For each partition $\bm{\lambda}\vdash n$, fix a tableau $\tau_{\bm{\lambda}}$ of shape $\bm{\lambda}$ 
and let $\{\mathsf{p}_1^{\tau_{\bm{\lambda}}}, \ldots, \mathsf{p}_{l_{\tau_{\bm{\lambda}}}}^{\tau_{\bm{\lambda}}}\}$ 
be a set of polynomials whose span contains $W_{\tau_{\bm{\lambda}}}$. 
Then for each partition $\bm{\lambda} \in \Lambda$, there exists a $l_{\tau_{\bm{\lambda}}} \times l_{\tau_{\bm{\lambda}}}$ psd matrix $R_{\bm{\lambda}}$ such that 
$$ \mathsf{p} = \sum_{\bm{\lambda} \in \Lambda} \textup{tr}( R_{\bm{\lambda}} \, Z^{\tau_{\bm{\lambda}}}),$$
where $Z^{\tau_{\bm{\lambda}}}_{ij} := \textup{sym}(\mathsf{p}_i^{\tau_{\bm{\lambda}}} \mathsf{p}_j^{\tau_{\bm{\lambda}}})$.
\end{theorem}

\proof
By Corollary~\ref{cor:GP with big lambda}, we know that under the hypotheses of
the theorem, there exist $m_{\bm{\lambda}}\times m_{\bm{\lambda}}$ real psd
matrices $Q_{\bm{\lambda}}$ such that $ \mathsf{p} = \sum_{\bm{\lambda}
\in \Lambda} \textup{tr}( Q_{\bm{\lambda}} \, Y^{\tau_{\bm{\lambda}}})$. Recall that the matrix
$Y^{\tau_{\bm{\lambda}}}$, defined as $Y^{\tau_{\bm{\lambda}}}_{ij} =
\textup{sym}(\mathsf{b}_i^{\tau_{\bm{\lambda}}}
\mathsf{b}_j^{\tau_{\bm{\lambda}}})$, is constructed from a
basis for $W_{\tau_{{\bm{\lambda}}}}$ after fixing one tableau
$\tau_{\bm{\lambda}}$ of shape $\bm{\lambda}$. Since each $\mathsf{b}_{i}^{\tau_{\bm{\lambda}}}$ is in the span of 
$\mathsf{p}_1^{\tau_{\bm{\lambda}}}, \ldots, \mathsf{p}_{l_{\tau_{\bm{\lambda}}}}^{\tau_{\bm{\lambda}}}$, there is a 
real matrix $M_{\tau_{\bm{\lambda}}}$ of size $l_{\tau_{\bm{\lambda}}} \times m_{\bm{\lambda}}$ such that 
$$(\textsf{p}_1^{\tau_{\bm{\lambda}}}, \ldots, \textsf{p}_{l_{\tau_{\bm{\lambda}}}}^{\tau_{\bm{\lambda}}})M_{\tau_{\bm{\lambda}}} = (\textsf{b}_1^{\tau_{\bm{\lambda}}}, \ldots, \textsf{b}_{m_{\bm{\lambda}}}^{\tau_{\bm{\lambda}}}).$$
Defining $R_{\bm{\lambda}} := M_{\tau_{\bm{\lambda}}}Q_{\bm{\lambda}} M_{\tau_{\bm{\lambda}}}^\top$, we see that 
\begin{align*}
(\textsf{b}_1^{\tau_{\bm{\lambda}}}, \ldots, \textsf{b}_{m_{\bm{\lambda}}}^{\tau_{\bm{\lambda}}}) Q_{\bm{\lambda}} (\textsf{b}_1^{\tau_{\bm{\lambda}}}, \ldots, \textsf{b}_{m_{\bm{\lambda}}}^{\tau_{\bm{\lambda}}})^\top 
= (\textsf{p}_1^{\tau_{\bm{\lambda}}}, \ldots, \textsf{p}_{l_{\tau_{\bm{\lambda}}}}^{\tau_{\bm{\lambda}}})R_{{\bm{\lambda}}}
(\textsf{p}_1^{\tau_{\bm{\lambda}}}, \ldots, \textsf{p}_{l_{\tau_{\bm{\lambda}}}}^{\tau_{\bm{\lambda}}})^\top.
\end{align*}
Symmetrizing both sides we conclude that 
$\textup{tr}(Q_{\bm{\lambda}} Y^{\tau_{\bm{\lambda}}}) = \textup{tr}(R_{{\bm{\lambda}}} \,Z^{\tau_{\bm{\lambda}}})$, and hence, 
\begin{align*}
\textsf{p}  = \sum_{\bm{\lambda} \in \Lambda} \textup{tr}( Q_{\bm{\lambda}} \, Y^{\tau_{\bm{\lambda}}})  
=  \sum_{\bm{\lambda} \in \Lambda} \textup{tr}( R_{\bm{\lambda}} \, Z^{\tau_{\bm{\lambda}}}).
\end{align*}
\qed

Note that the sos in Theorem~\ref{thm:GP with spanning set} retains the same structure as Corollary~\ref{cor:GP with big lambda} in the sense that
for each $\bm{\lambda}\in \Lambda$ we only need to consider $W_{\tau_{\bm{\lambda}}}$ for a single tableau $\tau_{\bm{\lambda}}$ of shape $\bm{\lambda}$.

% !TEX root =  main.tex
\section{Spanning sets and Razborov's flags}
\label{sec:spanning sets}

Recall that our goal is to find succinct sos expressions for $\fS_n$-invariant $d$-sos polynomials.
Theorem~\ref{thm:GP with spanning set} implies that to achieve this, one needs to find, 
for a fixed tableau $\tau_{\bm{\lambda}}$ of shape $\bm{\lambda}\in \Lambda$, 
a spanning set for $W_{\tau_{\bm{\lambda}}}$ with size independent of $n$.
In Section~\ref{sub:spanningsets}, we show how to construct a na\"ive spanning set for $W_{\tau_{\bm{\lambda}}}$ 
with the property that the number of polynomials in it is independent of $n$. From that first spanning set, 
we derive two more spanning sets that have the same property, and which naturally introduce the notion of flags 
introduced by Razborov in \cite{RazborovFlagAlgebras}. This connection to flags provides a nice combinatorial 
interpretation for these spanning sets which will be explained in Section~\ref{razrelation}.

\subsection{Spanning sets}\label{sub:spanningsets}
Recall that the vector space $W_{\tau_{\bm{\lambda}}}$ is a subspace of $\RR[\cV_n]_{\leq d}$ which is spanned by square-free monomials in $\binom{n}{2}$ variables
of degree at most $d$. We denote such a square-free monomial by $\mathsf{x}^m := \prod_{1\leq i<j\leq n} \mathsf{x}_{ij}^{m_{ij}}$ with $m_{ij}\in \{0,1\}$. 

A very na\"ive spanning set for $W_{\tau_{\bm{\lambda}}}$ can easily be generated as follows. For a tableau $\tau_{\bm{\lambda}}$ of shape $\bm{\lambda}=(\lambda_1, \lambda_2, \ldots)$ and a monomial $\mathsf{x}^m$, let $$\textup{sym}_{\tau_{\bm{\lambda}}}(\mathsf{x}^m) := \frac{1}{|\mathfrak{R}_{\tau_{\bm{\lambda}}}|}\sum_{\mathfrak{s} \in \mathfrak{R}_{\tau_{\bm{\lambda}}}} \mathfrak{s} \cdot \mathsf{x}^m$$ be the symmetrization of $\mathsf{x}^m$ under the row group  $\mathfrak{R}_{\tau_{\bm{\lambda}}}$. The polynomials $\textup{sym}_{\tau_{\bm{\lambda}}}(\mathsf{x}^m)$ as $\mathsf{x}^m$ varies over square-free monomials of degree at most $d$ form a natural spanning set for $W_{\tau_{\bm{\lambda}}}$. Indeed, recall that $W_{\tau_{\bm{\lambda}}}$ consists of $\mathfrak{R}_{\tau_{\bm{\lambda}}}$-invariant polynomials. Therefore, if $\mathsf{p}\in W_{\tau_{\bm{\lambda}}}$, then $\mathsf{p}=\textup{sym}_{\tau_{\bm{\lambda}}}(\mathsf{p})$ and so $\mathsf{p}$ is a linear combination of $\textup{sym}_{\tau_{\bm{\lambda}}}(\mathsf{x}^m)$ for different square-free monomials $\mathsf{x}^m$ of degree at most $d$.

We will show in Lemma~\ref{lemma:spansym} that the symmetrized monomials $\textup{sym}_{\tau_{\bm{\lambda}}}(\mathsf{x}^m)$ are superfluous when $\bm{\lambda}$ is not a hook.
Indeed, for any $\tau_{\bm{\lambda}}$, even when $\bm{\lambda}$ is not a hook, $W_{\tau_{\bm{\lambda}}}$ is in the span of symmetrizations of monomials under row groups of hooks.

\begin{definition}
Given a tableau $\tau_{\bm{\lambda}}$ of shape $\bm{\lambda}=(\lambda_1, \lambda_2, \ldots)$, define $\textup{hook}(\tau_{\bm{\lambda}})$ to be the tableau 
of shape $(\lambda_1, 1^{n-\lambda_1})$ such that the first row of $\textup{hook}(\tau_{\bm{\lambda}})$ is the same as that of $\tau_{\bm{\lambda}}$, and
the remaining labels in $\tau_{\bm{\lambda}}$ are put in the tail of the hook in increasing order.
\end{definition}

\begin{lemma}\label{lemma:spansym}
For the tableau $\tau_{\bm{\lambda}}$, the vector space $W_{\tau_{\bm{\lambda}}}$ is spanned by the polynomials $\textup{sym}_{\textup{hook}(\tau_{\bm{\lambda}})}(\mathsf{x}^m)$, as $\mathsf{x}^m$ varies over square-free monomials of degree at most $d$.
\end{lemma}

\proof

From the definition of $W_{\tau_{\bm{\lambda}}}$, we have that $W_{\tau_{\bm{\lambda}}} = V_{\bm{\lambda}}^{\fR_{\tau_{\bm{\lambda}}}}
\subseteq V^{\fR_{\tau_{\bm{\lambda}}}}$, the subspace of $\fR_{\tau_{\bm{\lambda}}}$-fixed elements of $V = \RR[\cV_n]_{\leq d}$. Since $\fR_{\textup{hook}(\tau_{\bm{\lambda}})}$
is a subgroup of $\fR_{\tau_{\bm{\lambda}}}$, it follows directly that $V^{\fR_{\tau_{\bm{\lambda}}}}\subseteq V^{\fR_{\textup{hook}(\tau_{\bm{\lambda}})}}$.
Hence $W_{\tau_{\bm{\lambda}}}\subseteq V^{\fR_{\textup{hook}(\tau_{\bm{\lambda}})}}$,
which is certainly spanned by the polynomials $\textup{sym}_{\textup{hook}(\tau_{\bm{\lambda}})}(\mathsf{x}^m)$,
as $\mathsf{x}^m$ varies over square-free monomials of degree at most $d$.
\qed

One can think about these hook polynomials in a combinatorial manner starting with the following definition.

\begin{definition}
The \emph{graph} of a square-free monomial $\mathsf{x}^m$, denoted as $G(\mathsf{x}^m)$, is the labeled subgraph of the complete graph $K_n$ containing the edge $\{i,j\}$ if $m_{ij} = 1$.
\end{definition}

Therefore, for any square-free monomial $\mathsf{x}^m$ and any tableau $\tau_{\bm{\lambda}}$ of shape $\bm{\lambda}=(\lambda_1, \lambda_2, \ldots)$, we have that

$$\textup{sym}_{\textup{hook}(\tau_{\bm{\lambda}})}(\mathsf{x}^m)=\frac{1}{|\mathfrak{R}_{\textup{hook}(\tau_{\bm{\lambda}})}|} \sum_{\sigma \in \mathfrak{R}_{\textup{hook}(\tau_{\bm{\lambda}})}} \prod_{\{i,j\} \in E(G(\mathsf{x}^m))} \mathsf{x}_{\sigma(i)\sigma(j)}.$$

Note that any vertex of $G(\mathsf{x}^m)$ that is in the tail of $\textup{hook}(\tau_{\bm{\lambda}})$ is fixed by $\mathfrak{R}_{\textup{hook}(\tau_{\bm{\lambda}})}$. This partitions the vertices of $G(\mathsf{x}^m)$ into two sets: those fixed by $\mathfrak{R}_{\textup{hook}(\tau_{\bm{\lambda}})}$, and those that are not. We now see how this phenomenon is mirrored by flags, partially labeled graphs where the labeled part is fixed. We can thus view graphs $G(\mathsf{x}^m)$ through the lens of flags and show that $\textup{sym}_{\textup{hook}(\tau_{\bm{\lambda}})}(\mathsf{x}^m)$ can be thought of as polynomials coming from flags instead. To formalize this connection, we introduce flags and density polynomials arising from flags and then show their equivalence to the symmetric polynomials defined above.

\begin{definition}\label{def:flags}
Consider three fixed integers $0 \leq t \leq f \leq n$.
\begin{enumerate}
\item An {\em intersection type} of size $t$ is a simple graph $T$ on $t$ vertices
in which every vertex is labeled with a distinct element of
$[t]$.
\item A  $T$-{\em
  flag} $F$ of size $f$ is a simple graph on $f$ vertices with $t$ of its
vertices labeled $1, \ldots, t$ such that these labeled vertices induce
a copy of $T$ in $F$ with identical labels for the vertices.
\item Let $\mathcal{F}^f_T$ be the set of all $T$-flags of size $f$ up
to isomorphism.
\end{enumerate}
\end{definition}

\begin{example}
Consider $T=\labeledcherries{1}{2}{3}$ and $f=4$. Then $\mathcal{F}_T^f$ consists of the $T$-flags:
\begin{align*}
 &F_0=\flagzero{1}{2}{3}{}, F_1=\flagone{1}{2}{3}{}, F_2=\flagtwo{1}{2}{3}{}, F_3=\flagthree{1}{2}{3}{}, \\
&F_4=\flagfour{1}{2}{3}{}, F_5=\flagfive{1}{2}{3}{}, F_6=\flagsix{1}{2}{3}{}, F_7=\flagseven{1}{2}{3}{}.
\end{align*}
\end{example}

\begin{definition}
For two sets $A$ and $B$, let $\textup{Inj}(A,B)$ denote the set of injective maps $h: A\rightarrow B$.
Suppose $F$ is a $T$-flag of size $f$ and  $\Theta \in \textup{Inj}([t],[n])$.
We say that $h \in \textup{Inj}(V(F),[n])$ respects the labeling $\Theta$ if $h(v)=\Theta(i)$ for any vertex $v\in V(F)$ labeled $i\in [t]$. Let $\textup{Inj}_{\Theta}(V(F),[n])$ denote the set of all injective maps $h: V(F) \rightarrow [n]$ that respect the labeling $\Theta$.
\end{definition}

\begin{example}
Suppose $T=\labeledcherries{1}{2}{3}$ and $F=\flagtwo{1}{2}{3}{}$. Identify $[n]$ with $\{a_1,a_2,a_3,\ldots\}$. Fix $\Theta \in \textup{Inj}([3],[n])$ to be the map
$\Theta(1)=a_2$, $\Theta(2)=a_1$ and $\Theta(3)=a_4$. Let $v$ denote the unlabeled vertex in $F$.
Then $h \in \textup{Inj}(V(F),[n])$ such that $h(1)=a_2$, $h(2)=a_1$, $h(3)=a_4$, $h(v)\in [n]\setminus\{a_1,a_2,a_4\}$ is $\Theta$-preserving, while
$h$ such that $h(1)=a_3$, $h(2)=a_1$, $h(3)=a_4$, $h(v)=a_2$ is not.
\end{example}

\begin{definition}
Fix $T, f, \Theta$. Then for every flag $F\in \mathcal{F}_{T}^f$, we let $$\mathsf{g}_F^\Theta=\sum_{h\in \textup{Inj}_{\Theta}(V(F),[n])}\prod_{\{i,j\}\in E(F)} \mathsf{x}_{h(i)h(j)}.$$
\end{definition}

Note that $\mathsf{g}_F^\Theta$ is always square-free since $F$ is assumed to be a simple graph. 
So the polynomials $\mathsf{g}_F^\Theta$ cannot span $\mathbb{R}[\mathsf{x}]_{\leq d}$
but can possibly span $\RR[\cV_n]_{\leq d}$.  

\begin{example}\label{ex:g}
Consider our running example with the labeling $\Theta$ fixed to be  $\Theta(1)=a_i$, $\Theta(2)=a_j$ and $\Theta(3)=a_k$ where $a_i,a_ j,a_k \in [n]$ are distinct and fixed. \\

\begin{center}
\begin{tabular}{|c|l|}
\hline
$F$ & $\mathsf{g}_F^\Theta$\\
\hline
$F_0=\flagzero{1}{2}{3}{}$ & $(n-3)\mathsf{x}_{ij}\mathsf{x}_{ik}$\\
$F_1=\flagone{1}{2}{3}{}$ & $\mathsf{x}_{ij}\mathsf{x}_{ik}\sum_{l\in [n] \backslash\{i,j,k\}} \mathsf{x}_{il}$\\
$F_4=\flagfour{1}{2}{3}{}$ & $\mathsf{x}_{ij}\mathsf{x}_{ik}\sum_{l\in [n] \backslash\{i,j,k\}} \mathsf{x}_{il} \mathsf{x}_{jl}$\\
$F_7=\flagseven{1}{2}{3}{}$ & $\mathsf{x}_{ij}\mathsf{x}_{ik}\sum_{l\in [n] \backslash\{i,j,k\}} \mathsf{x}_{il} \mathsf{x}_{jl}\mathsf{x}_{kl}$\\
\hline
\end{tabular}
\end{center}
\end{example}

The polynomials $\mathsf{g}_F^\Theta$ can be interpreted in terms of tableaux and their row groups.

\begin{definition}
For $\Theta \in \textup{Inj}([t],[n])$, let $\textup{hook}^\Theta$ be the tableau of shape $(n-t,1^t)$ 
where the labels in $\Theta([t])$  are placed in increasing order down the tail and the remaining labels in $[n]$ 
are put in increasing order in the first row.
\end{definition}

\begin{example}
Let $\Theta$ be as in Example~\ref{ex:g} and assume  $a_i < a_j < a_k$. Then
$$\textup{hook}^\Theta=\Yboxdimx{23pt}\young(<a_{l_1}><a_{l_2}>\cdots<a_{l_{n-3}}>,<a_i>,<a_j>,<a_k>)$$
where $a_{l_r} < a_{l_{r+1}}$ for all $r$, and $\{a_{l_1}, \ldots, a_{l_{n-3}} \}=[n]\backslash \{a_i,a_j,a_k\}$.
\end{example}

Since $\mathfrak{R}_{\textup{hook}^\Theta}$ can be identified with the symmetric group on the labels $[n]\backslash \Theta([t])$, we immediately get the following connection between 
$\mathsf{g}_F^\Theta$ and  $\textup{sym}_{\textup{hook}^\Theta}(\mathsf{x}^m)$.

\begin{lemma} \label{lem:g = sym hook}
Fix any $F\in \mathcal{F}_T^f$ and $\Theta\in \textup{Inj}([t],[n])$. Then the polynomial $\mathsf{g}_F^\Theta$ is equal to $\frac{(n-t)!}{(n-f)!}\cdot
\textup{sym}_{\textup{hook}^\Theta}(\mathsf{x}^m)$ where
$\mathsf{x}^m=\prod_{\{i,j\}\in E(F)} \mathsf{x}_{h^*(i)h^*(j)}$ for any fixed $h^*\in \textup{Inj}_{\Theta}(V(F),[n])$.
\end{lemma}

\proof
This follows from the fact that $|\textup{Inj}_\Theta(V(F),[n])|=\binom{n-t}{f-t}(f-t)!$ and that, for each $h\in \textup{Inj}_\Theta(V(F),[n])$, there are $(n-f)!$ permutations $\sigma\in \mathfrak{R}_{\textup{hook}^\Theta}$ such that $\sigma(h(\mathsf{x}^m))=h(\mathsf{x}^m)$. 
\qed

\begin{example}
 From the previous example, we have that
 $$\mathsf{g}_{F_0}^\Theta= (n-3) \cdot \textup{sym}_{\textup{hook}^\Theta}(\mathsf{x}_{ij}\mathsf{x}_{ik}), \,\,\,\,
 \mathsf{g}_{F_1}^\Theta=(n-3) \cdot \textup{sym}_{\textup{hook}^\Theta}(\mathsf{x}_{ij}\mathsf{x}_{ik}\mathsf{x}_{i l_1}).$$
\end{example}

Therefore, we will be able to reinterpret our na\"ive spanning set for a particular $W_{\tau_{\bm{\lambda}}}$ in terms of flags. Before doing so, we introduce a second set of flag-based polynomials. Applying a  M\"obius transformation \cite[A.1]{LovaszBook} on the ${g_F^\Theta}$'s, we get a new set of polynomials. In Subsection \ref{razrelation}, we remark on their link to flag algebras and graph homomorphism densities.

\begin{definition}
\begin{enumerate}

\item Define
\begin{align*}
\mathcal{F}_{\geq T}^f := \bigcup \bigg\{F \,:\, & F \textup{ is a } T^\prime\textup{-flag of size } f
 \textup{ where the vertices } T' \textup{are the same as}\\
& \textup{those of } T \textup{ (including the labels), and } E(T') \supseteq E(T) \bigg\},
\end{align*}
and let $P_T^f$ be the poset on $\mathcal{F}_{\geq T}^f$ where $F\leq F'$ if $E(F)\subseteq E(F')$. This poset can be graded by the number of edges.

\item For a flag $F\in \mathcal{F}_{\geq T}^f$, define the polynomial
$$\mathsf{d}_F^\Theta :=\sum_{\substack{F'\in P_T^f:\\ F' \geq F}} (-1)^{|E(F')|-|E(F)|} \mathsf{g}_{F'}^\Theta.$$

\end{enumerate}
\end{definition}

Note that unlike $\mathcal{F}_{T}^f$, the set
$\mathcal{F}_{\geq T}^f$ is not defined up to isomorphism; for example, it would contain both \Fones{1} and \Fone{1} even though they are isomorphic flags. 

\begin{example} \label{ex:d theta F}
Consider $T$ and $f$ as before. They yield the following poset $P_T^f$.
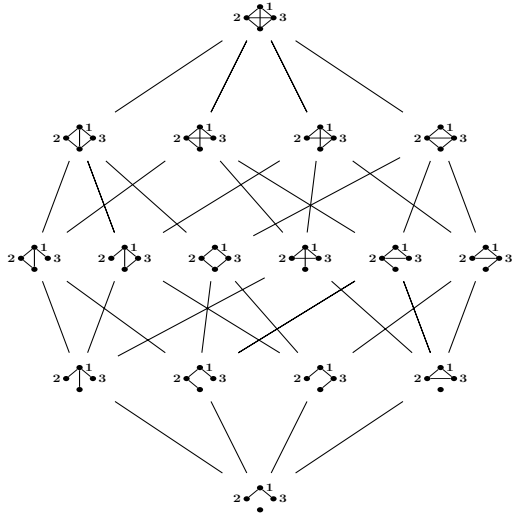
\begin{center}
\begin{tikzpicture}[scale=0.8]
\node (bottom) at (0,0) {\flagzero{1}{2}{3}{}};
\node (Fone) at (-3,2) {\flagone{1}{2}{3}{}};
\node (Ftwo) at (-1,2) {\flagtwo{1}{2}{3}{}};
\node (Fthree) at (1,2) {\flagthree{1}{2}{3}{}};
\node (Gzero) at (3,2) {\gzero{1}{2}{3}{}};
\node (Ffour) at (-3.75,4) {\flagfour{1}{2}{3}{}};
\node (Ffive) at (-2.25,4) {\flagfive{1}{2}{3}{}};
\node (Fsix) at (-0.75, 4) {\flagsix{1}{2}{3}{}};
\node (Gone) at (0.75, 4) {\gone{1}{2}{3}{}};
\node (Gtwo) at (2.25, 4) {\gtwo{1}{2}{3}{}};
\node (Gthree) at (3.75, 4) {\gthree{1}{2}{3}{}};
\node (Fseven) at (-3,6) {\flagseven{1}{2}{3}{}};
\node (Gfour) at (-1, 6) {\gfour{1}{2}{3}{}};
\node (Gfive) at (1,6) {\gfive{1}{2}{3}{}};
\node (Gsix) at (3,6) {\gsix{1}{2}{3}{}};
\node (Gseven) at (0,8) {\gseven{1}{2}{3}{}};

\draw (bottom)--(Fone);
\draw (bottom)--(Ftwo);
\draw (bottom)--(Fthree);
\draw (bottom)--(Gzero);
\draw (Fone)--(Ffour)--(Ftwo)--(Fsix)--(Fthree)--(Ffive)--(Fone)--(Gone)--(Gzero)--(Gtwo)--(Ftwo)--(Gtwo)--(Gzero)--(Gthree)--(Fthree);
\draw (Ffour)--(Fseven)--(Ffive)--(Fseven)--(Fsix)--(Gsix)--(Gthree)--(Gfive);
\draw (Gfour)--(Gtwo)--(Gsix);
\draw (Ffour)--(Gfour)--(Gone)--(Gfive)--(Ffive);
\draw (Fseven)--(Gseven)--(Gfour)--(Gseven)--(Gfive)--(Gseven)--(Gsix);
\end{tikzpicture}
\end{center}

Continuing our example, for $F_6 = \flagsix{1}{2}{3}{}$, the polynomial
$$\mathsf{d}_{F_6}^\Theta=
\mathsf{x}_{ij}\mathsf{x}_{ik}
\sum_{l\in [n] \backslash\{i,j,k\}} \left(
\mathsf{x}_{jl} \mathsf{x}_{kl}
- \mathsf{x}_{il} \mathsf{x}_{jl}\mathsf{x}_{kl}
-\mathsf{x}_{jk}\mathsf{x}_{jl} \mathsf{x}_{kl}+\mathsf{x}_{jk}\mathsf{x}_{il} \mathsf{x}_{jl}\mathsf{x}_{kl} \right).$$
\end{example}

\begin{definition}
Given a tableau $\tau_{\bm{\lambda}}$ of shape $\bm{\lambda}=(\lambda_1, \lambda_2, \ldots)$, define $\Theta_{\tau_{\bm{\lambda}}}\in \textup{Inj}([n-\lambda_1],[n])$ such that $\Theta_{\tau_{\bm{\lambda}}}(i)$ is the $i$th smallest label not in the first row of $\tau_{\bm{\lambda}}$.
\end{definition}

\begin{lemma}\label{extravertices}
Let $F\in \mathcal{F}_T^f$ and $F'\in \mathcal{F}_T^{f'}$ where $n\geq f'>f$ and $E(F)=E(F')$, i.e., $F'$ is obtained from $F$ by adding isolated vertices. Then $$\mathsf{g}^\Theta_{F'}= \binom{n-f}{f'-f}(f'-f)! \cdot \mathsf{g}^\Theta_F.$$
\end{lemma}

\proof
Since $E(F)=E(F')$,

\begin{align*}
\mathsf{g}_{F'}^\Theta&=\sum_{h\in \textup{Inj}_{\Theta}(V(F'),[n])}\prod_{\{i,j\}\in E(F')} \mathsf{x}_{h(i)h(j)}\\
&=\sum_{h\in \textup{Inj}_{\Theta}(V(F'),[n])}\prod_{\{i,j\}\in E(F)} \mathsf{x}_{h(i)h(j)}\\
&=\binom{n-f}{f'-f}(f'-f)! \sum_{h\in \textup{Inj}_{\Theta}(V(F),[n])}\prod_{\{i,j\}\in E(F)} \mathsf{x}_{h(i)h(j)}.\qedhere
\end{align*}

We will now show that for the vector space $W_{\tau_{\bm{\lambda}}}$, the polynomials $g_F^{\Theta_{\tau_{\bm{\lambda}}}}$ and $d_F^{\Theta_{\tau_{\bm{\lambda}}}}$ form a
spanning set as $F$ varies over all graphs in  $\mathcal{F}_T^{2d}$ where $T$ has $n-\lambda_1$ vertices. From now on, we denote the number of vertices in $T$ by $|T|$.

\begin{theorem}\label{thm:span}
For the tableau $\tau_{\bm{\lambda}}$, the vector space $W_{\tau_{\bm{\lambda}}}$ is spanned by
\begin{enumerate}
\item the polynomials $\mathsf{g}_F^{\Theta_{\tau_{\bm{\lambda}}}}$ with flags $F\in \mathcal{F}_T^{2d}$ where $|T|=n-\lambda_1$, and

\item  the polynomials $\mathsf{d}_F^{\Theta_{\tau_{\bm{\lambda}}}}$ with flags $F\in \mathcal{F}_T^{2d}$ where $|T|=n-\lambda_1$.

\end{enumerate}
\end{theorem}

\proof
Recall from Lemma~\ref{lemma:spansym} that the polynomials $\textup{sym}_{\textup{hook}(\tau_{\bm{\lambda}})}(\mathsf{x}^m)$ span $W_{\tau_{\bm{\lambda}}}$ as $\mathsf{x}^m$ varies over certain square-free monomials of degree at most $d$.

Since $W_{\tau_{\bm{\lambda}}}$ sits in the span of the polynomials $\textup{sym}_{\textup{hook}(\tau_{\bm{\lambda}})}(\mathsf{x}^m)$, it also sits in the span of $\frac{|\mathfrak{R}_{\textup{hook}(\tau_{\bm{\lambda}})}|}{(n-|V(G(\mathsf{x}^m))|)!}\cdot \textup{sym}_{\textup{hook}(\tau_{\bm{\lambda}})}(\mathsf{x}^m)$,  where $\mathsf{x}^m$ varies over square-free monomials of degree at most $d$. For a polynomial $\frac{|\mathfrak{R}_{\textup{hook}(\tau_{\bm{\lambda}})}|}{(n-|V(G(\mathsf{x}^m))|)!}\cdot\textup{sym}_{\textup{hook}(\tau_{\bm{\lambda}})}(\mathsf{x}^m)$, consider the
graph $G(\mathsf{x}^m)$ and note that, since $\mathsf{x}^m$ has degree at most $d$, $G(\mathsf{x}^m)$ has at most $2d$ vertices. Unlabel any vertex in $G(\mathsf{x}^m)$ whose label is in $\textup{row}_1(\textup{hook}(\tau_{\bm{\lambda}}))$, and replace the labels of the remaining $n-\lambda_1$ labeled vertices with the labels in $[n-\lambda_1]$ through a map $\phi$ that does so in increasing order (i.e., the smallest labeled vertex becomes vertex $1$, and the largest labeled vertex becomes vertex $n-\lambda_1$). Let $F$ be that graph, and let $\Theta=\phi^{-1}$. Note that $|V(F)|=|V(G(\mathsf{x}^m))|$ and $\Theta=\Theta_{\tau_{\bm{\lambda}}}$. Then, by Lemma~\ref{lem:g = sym hook}, $\mathsf{g}_F^\Theta=\frac{|\mathfrak{R}_{\textup{hook}(\tau_{\bm{\lambda}})}|}{(n-|V(G(\mathsf{x}^m))|)!}\cdot \textup{sym}_{\textup{hook}(\tau_{\bm{\lambda}})}(\mathsf{x}^m)$, and so $W_{\tau_{\bm{\lambda}}}$ sits in the span of $\mathsf{g}_F^\Theta$ for flags $F\in \mathcal{F}_T^f$ where $f\leq 2d$, $|T|=n-\lambda_1$ and $\Theta=\Theta_{\tau_{\bm{\lambda}}}$. See Example~\ref{eg:process} for 
an illustration of the process just described.

By Lemma \ref{extravertices}, one only needs to keep the polynomials $\mathsf{g}_F^\Theta$ for flags $F\in \mathcal{F}_T^f$ where $f=2d$ and $|T|=n-\lambda_1$ since, for any flag $F^*\in \mathcal{F}_T^{f^*}$ where $f^*<2d$, there exists a flag $F^+$ with the same edge set but on $2d$ vertices such that $\mathsf{g}^\Theta_{F^+}=k\mathsf{g}^\Theta_{F^*}$ where $k$ is a scalar.

The spans of polynomials $\mathsf{g}_{F}^\Theta$ and $\mathsf{d}_{F}^\Theta$ for $F\in \mathcal{F}_{\leq T}^{2d}$ coincide since they are related by a M\"obius transformation.
Thus $W_{\tau_{\bm{\lambda}}}$ also sits in the span of the polynomials $\mathsf{d}_F^\Theta$ for flags $F\in \mathcal{F}_T^{2d}$ with $|T|=n-\lambda_1$ and where $\Theta=\Theta_{\tau_{\bm{\lambda}}}$.
\qed

\begin{remark}
Let $\bm{\lambda}, \bm{\lambda}'$ be partitions of $n$ with ${\lambda}_1 = {\lambda}'_1$, and let $\tau_{\bm{\lambda}}$ and $\tau_{\bm{\lambda}'}$ be two tableaux of shape $\bm{\lambda}$ and $\bm{\lambda}'$ such that  $\textup{row}_1(\tau_{\bm{\lambda}}) = \textup{row}_1(\tau_{\bm{\lambda}'})$ and thus $\textup{hook}(\tau_{\bm{\lambda}})=\textup{hook}(\tau_{\bm{\lambda}'})$. Then, since ${\Theta_{\tau_{\bm{\lambda}}}} = {\Theta_{\tau_{\bm{\lambda}'}}}$, the vector spaces $W_{\tau_{\bm{\lambda}}}$ and $W_{\tau_{\bm{\lambda}'}}$ are both spanned by 
the polynomials $\mathsf{g}_F^{\Theta_{\tau_{\bm{\lambda}}}}$ (respectively  $\mathsf{d}_F^{\Theta_{\tau_{\bm{\lambda}}}}$ ) with flags $F\in \mathcal{F}_T^{2d}$ where $|T|=n-\lambda_1$. In particular, the 
polynomials in Theorem~\ref{thm:span} that span 
$W_{\textup{hook}(\tau_{\bm{\lambda}})}$ also span $W_{\tau_{\bm{\lambda}}}$. 
\end{remark}

\begin{example}
\label{eg:process}
 We give an example to illustrate the proof of the theorem above. Consider any tableau $\tau_{\bm{\lambda}}$ for which

$$\textup{hook}(\tau_{\bm{\lambda}}) = \Yboxdimx{23pt}\young(<a_{l_1}><a_{l_2}>\cdots<a_{l_{n-3}}>,<a_i>,<a_j>,<a_k>)$$
where $a_i < a_j < a_k$ are fixed and $\{a_{l_1}, \ldots, a_{l_{n-3}}\}=[n]\backslash\{a_i, a_j, a_k\}$. Now consider the monomial $\mathsf{x}^m = \mathsf{x}_{ij}\mathsf{x}_{ik}\mathsf{x}_{jl_r}\mathsf{x}_{kl_r}$ for some fixed $l_r$ such that $1\leq r \leq n-3$.
Then note that
$$\frac{|\mathfrak{R}_{\textup{hook}(\tau_{\bm{\lambda}})}|}{(n-|V(G(\mathsf{x}^m))|)!}\cdot \textup{sym}_{\textup{hook}(\tau_{\bm{\lambda}})}(\mathsf{x}^m) =\mathsf{x}_{ij} \mathsf{x}_{ik} \sum_{l \in [n] \backslash \{i,j,k\}} x_{jl}x_{kl}.$$

The graph $G(\mathsf{x}^m)$ is \flagsix{$\bm{i}$}{$\bm{j}$}{$\bm{k}$}{$\bm{l}_r$}. Unlabeling the vertex labeled $l_r$ since $l_r$ is in the first row of $\textup{hook}(\tau_{\bm{\lambda}})$, we get the graph
$\flagsix{$\bm{i}$}{$\bm{j}$}{$\bm{k}$}{}$. Now we replace the labels $i,j,k$ with $1,2,3$ in that order to get $F_6 = \flagsix{$1$}{$2$}{$3$}{}$. Thus the map $\phi$ in the proof of the lemma sends $i \mapsto 1, j \mapsto 2, k \mapsto 3$. Hence take
$\Theta$ to be the map that sends $1 \mapsto i, 2 \mapsto j, 3 \mapsto k$ and $F$ to be the flag $F_6$.
Then the polynomial $\mathsf{g}^{\Theta}_{F_6}$ is exactly
$(n-3)\cdot \textup{sym}_{\textup{hook}(\tau_{\bm{\lambda}})}(\mathsf{x}^m)$.
\end{example}

We now observe that, when $\bm{\lambda}\geq_{\textup{lex}} (n-2d, 1^{2d})$, we have been successful in generating spanning sets for $W_{\tau_{\bm{\lambda}}}$ whose sizes are independent of $n$.

\begin{proposition}
\label{prop:mlambda-indn}
Fix some tableau $\tau_{\bm{\lambda}}$ of shape $\bm{\lambda}\geq_{\textup{lex}} (n-2d,1^{2d})$. Then the number of polynomials in each spanning set of $W_{\tau_{\bm{\lambda}}}$ described in Lemma~\ref{lemma:spansym} and Theorem~\ref{thm:span} is independent of $n$.
\end{proposition}

\proof
Suppose $\bm{\lambda}=(\lambda_1, \lambda_2, \ldots)$. Then, since $\bm{\lambda}\geq_{\textup{lex}} (n-2d,1^{2d})$, we have that $n-\lambda_1 \leq 2d$. Thus, the polynomials $\mathsf{g}^{\Theta_{\tau_{\bm{\lambda}}}}_F$ (or $\mathsf{d}^{\Theta_{\tau_{\bm{\lambda}}}}_F$) with flags $F\in \mathcal{F}^{2d}_T$ span $W_{\tau_{\bm{\lambda}}}$ as $T$ varies over all intersection types on $n-\lambda_1\leq 2d$ vertices. Note that the number of intersection types on at most $2d$ vertices does not depend on $n$. Moreover, note that the number of flags in $\mathcal{F}_T^{2d}$ for any intersection type $T$ such that $|T|\leq 2d$ is independent of $n$ by Definition \ref{def:flags}. Thus the number of polynomials $\mathsf{g}^{\Theta_{\tau_{\bm{\lambda}}}}_F$ (or $\mathsf{d}^{\Theta_{\tau_{\bm{\lambda}}}}_F$) does not depend on $n$. By the bijection in Lemma~\ref{lem:g = sym hook}, the result holds.
\qed

\subsection{Relationship with flag and graph densities}\label{razrelation}
We now discuss how to interpret combinatorially the spanning sets involving $\mathsf{g}^{\Theta_{\tau_{\bm{\lambda}}}}_F$'s or $\mathsf{d}^{\Theta_{\tau_{\bm{\lambda}}}}_F$'s.

The {\em characteristic vector } of a graph $G$, denoted as $\mathds{1}_G$, is the vector in $\{0,1\}^{{n \choose 2}}$ whose $ij$-entry is one if and only if $\{i,j\} \in E(G)$. The polynomial $\mathsf{g}^\Theta_F$, when evaluated on $\mathds{1}_G$ for  a graph $G$ on $n$ vertices, calculates the number ways of picking uniformly at random a set of $t$ vertices of $G$ to label by $\Theta$ and $f-t$ additional unlabeled vertices in such a way that they produce a copy of $F$ in $G$ (not necessarily as an induced subgraph). These numbers are at the core of the theory of graph homomorphisms.
In fact, the polynomials $\mathsf{g}^\Theta_F$ and $\mathsf{d}^\Theta_F$, and
the relationship between them, exactly matches the graph densities
$t_{\textup{inj}}$ and $t_{\textup{ind}}$ in the theory of graph
limits~\cite[Chapter 5]{LovaszBook}.

On the other hand, we now observe that $\mathsf{d}_F^\Theta(\mathds{1}_G)$ for a  graph $G$ on $n$ vertices,
is the number of ways in which one can induce a copy of $F$ in $G$ by picking uniformly at random $t$ vertices in $G$ to label by $\Theta$ along with $f-t$ unlabeled vertices i.e., $\binom{n-t}{f-t} (f-t)!$ times the {\em flag density of $F$ in $G$}, as in the theory of flag algebras \cite{RazborovFlagAlgebras}. This brings us to a combinatorial
 expression for the polynomial $\mathsf{d}^\Theta_F$.

\begin{lemma}
The polynomial
$$\mathsf{d}^\Theta_F=\sum_{h \in\textup{Inj}_\Theta(V(F),[n])}  \prod_{\{i,j\}\in E(F)}\mathsf{x}_{h(i)h(j)}\prod_{\{i,j\}\in \binom{V(F)}{2}\backslash E(F)}(1-\mathsf{x}_{h(i)h(j)}).$$
\end{lemma}

\proof
Expanding the expression on the right-hand side we obtain
$$\sum_{h\in \textup{Inj}_{\Theta}(V(F),[n])}\sum_{F'\supseteq F} a_{F'}\prod_{\{i,j\}\in E(F')} \mathsf{x}_{h(i)h(j)},$$
where $a_{F'}=(-1)^{|E(F')|-|E(F)|}\cdot|\{F''|F'' \textup{ is isomorphic to } F'\}|$  and $F' \supseteq  F$ is a graph on the labeled and unlabeled vertices of $F$ and where any edge in $F$ is also an edge in $F'$. In other words,
$F' \in \mathcal{F}_{\geq T}^f$ where $T$ is the intersection type of $F$ and $f = | V(F) |$, and the result holds.
\qed

\begin{example}
Recall from Example~\ref{ex:d theta F} that $$\mathsf{d}_{F_6}^\Theta=
\mathsf{x}_{ij}\mathsf{x}_{ik}
\sum_{l\in [n] \backslash\{i,j,k\}} \left(
\mathsf{x}_{jl} \mathsf{x}_{kl}
- \mathsf{x}_{il} \mathsf{x}_{jl}\mathsf{x}_{kl}
-\mathsf{x}_{jk}\mathsf{x}_{jl} \mathsf{x}_{kl}+\mathsf{x}_{jk}\mathsf{x}_{il} \mathsf{x}_{jl}\mathsf{x}_{kl} \right).$$
However, this expression can be rewritten to see that
$$
\mathsf{d}_{F_6}^\Theta=\sum_{l\in [n] \backslash\{i,j,k\}} \mathsf{x}_{ij}\mathsf{x}_{ik} \mathsf{x}_{jl} \mathsf{x}_{kl} (1- \mathsf{x}_{il})
(1-\mathsf{x}_{jk}).$$

\end{example}

% !TEX root =  main.tex
\section{Flag sums of squares for the hypercube}
\label{sec:razborov sos}

Applying Theorem \ref{thm:GP with spanning set} to the spanning sets for $W_{\tau_{\bm{\lambda}}}$ consisting of the
$\mathsf{d}^\Theta_F$ and $\mathsf{g}^\Theta_F$ polynomials from the previous section, we immediately get the following corollaries.

\begin{corollary}\label{cor:gsumsofsquares}
Suppose $\mathsf{p}$ is symmetric and $d$-sos.
For each partition $\bm{\lambda} \in \Lambda$, fix a tableau $\tau_{\bm{\lambda}}$ of shape $\bm{\lambda}$. Then there exists psd matrices $R_{\bm{\lambda}}$ such that
$$ \mathsf{p} = \sum_{\bm{\lambda} \in \Lambda} \textup{tr}( R_{\bm{\lambda}} \, Z^{\tau_{\bm{\lambda}}}),$$
where $Z^{\tau_{\bm{\lambda}}} := \textup{sym}(\bm{\mathsf{g}}_{\tau_{\bm{\lambda}}}\bm{\mathsf{g}}_{\tau_{\bm{\lambda}}}^\top )$ and $\bm{\mathsf{g}}_{\tau_{\bm{\lambda}}}$ is the vector of polynomials $\mathsf{g}^{\Theta_{\tau_{\bm{\lambda}}}}_F$ such that $F\in \mathcal{F}_T^{2d}$ where $|T|=n-\lambda_1$.
\end{corollary}

\begin{corollary}\label{cor:dsumsofsquares}
Suppose $\mathsf{p}$ is symmetric and $d$-sos.
For each partition $\bm{\lambda} \in \Lambda$, fix a tableau $\tau_{\bm{\lambda}}$ of shape $\bm{\lambda}$. Then there exists psd matrices $R_{\bm{\lambda}}$ such that
$$ \mathsf{p} =  \sum_{\bm{\lambda} \in \Lambda} \textup{tr}( R_{\bm{\lambda}} \, Z^{\tau_{\bm{\lambda}}}),$$
where $Z^{\tau_{\bm{\lambda}}} := \textup{sym}(\bm{\mathsf{d}}_{\tau_{\bm{\lambda}}}\bm{\mathsf{d}}_{\tau_{\bm{\lambda}}}^\top )$ and $\bm{\mathsf{d}}_{\tau_{\bm{\lambda}}}$ is the vector of polynomials $\mathsf{d}^{\Theta_{\tau_{\bm{\lambda}}}}_F$ such that $F\in \mathcal{F}_T^{2d}$ where $|T|=n-\lambda_1$.
\end{corollary}

By Proposition~\ref{prop:mlambda-indn}, the size of each $R_{\bm{\lambda}}$, which equals the size of the spanning set being used for $W_{\tau_{\bm{\lambda}}}$,  is independent of $n$. By Proposition~\ref{prop:total number of m lambdas}, the number of partitions $\bm{\lambda} \in \Lambda$ is also independent of $n$. Therefore, the sizes of the SDPs that need to be solved in Corollaries~\ref{cor:gsumsofsquares} and \ref{cor:dsumsofsquares} are independent of $n$.
Thus, we have established that symmetric non-negative polynomials over the discrete hypercube $\mathcal{V}_n$ have succinct sos certificates that come from the flag polynomials $\mathsf{d}^\Theta_F$ and their cousins $\mathsf{g}^\Theta_F$. In particular,
flag algebra techniques extend beyond the realm of extremal graph theory and can be used to establish
symmetric inequalities on $\mathcal{V}_n$.

The sos expressions that typically appear in the flag algebra literature, e.g., \cite{Razborov10}, \cite{Razborov14}, \cite{Falgas-RavryVaughan}, are more restricted than those in Corollary~\ref{cor:dsumsofsquares}.
The main result of this section is to show that even these rather special looking sos expressions, which we will refer to as {\em flag sos} expressions, are sufficient to establish the non-negativity of symmetric polynomials over $\mathcal{V}_n$. The first step in passing to these restricted sos expressions is the following observation.

\begin{lemma} \label{lem:orthogonality}
For two flags $F$ and $F'$ with different intersection types of equal size $t$, and any labeling $\Theta \in \textup{Inj}([t],[n])$, $\mathsf{d}_F^\Theta\mathsf{d}_{F'}^\Theta$
is the zero function on $\mathcal{V}_n$.
\end{lemma}

\proof
Suppose $F$ is a $T$-flag, $F'$ is a $T'$-flag and $|E(T)|\geq |E(T')|$. Since $T \neq T'$, there exist
vertices $i$ and $j$ such that $\{i,j\}\in E(T)$ and $\{i,j\}\not\in E(T')$.
Therefore, there is some polynomial $\mathsf{q}$ such that
$$\mathsf{d}_F^\Theta\mathsf{d}_{F'}^\Theta=\mathsf{x}_{ij}(1-\mathsf{x}_{ij})\mathsf{q}=0.
$$
The last equality follows from the fact that $\mathsf{x}_{ij}(1-\mathsf{x}_{ij}) = 0$ on $\mathcal{V}_n$.
\qed

Note that the hypercube ideal $\mathcal{I}_n$ was crucial for the proof of Lemma~\ref{lem:orthogonality}.
Using this lemma, we now obtain a refined version of Corollary~\ref{cor:dsumsofsquares} (and similarly, Corollary~\ref{cor:gsumsofsquares}, which we omit). This will allow us to translate our $d$-sos expressions into flag sums of squares afterwards.

 \begin{proposition} \label{prop:new dsumsofsquares}
 Suppose $\mathsf{p}$ is symmetric and $d$-sos. For each intersection type $T$ of size $t \leq 2d$, fix $\Theta_T \in \textup{Inj}([t],[n])$.
 Then there exists psd matrices $R_T$ such that
 $$\mathsf{p} =  \sum_{t=0}^{2d} \sum_{T : |T| = t}
\textup{tr}\left(R_{T} \,\sym(\bm{\mathsf{d}}_{T} \bm{\mathsf{d}}_{T}^\top)\right),$$
where $\bm{\mathsf{d}}_T$ is the vector of polynomials
$\bm{\mathsf{d}}_F^{\Theta_T}$ such that $F \in \mathcal{F}_{T}^{2d}$.
 \end{proposition}

\proof Suppose $\mathsf{p}$ is symmetric and $d$-sos.
For each partition $\bm{\lambda} \in \Lambda$ (where $\Lambda=\{\bm{\lambda}|\bm{\lambda}\geq_{\textup{lex}}(n-2d,1^{2d})\}$), fix a tableau $\tau_{\bm{\lambda}}$ of shape $\bm{\lambda}=(\lambda_1,\lambda_2, \ldots)$. Then, by Corollary \ref{cor:dsumsofsquares}, there exist psd matrices $R_{\bm{\lambda}}$ such that
$$ \mathsf{p} =  \sum_{\bm{\lambda} \in \Lambda} \textup{tr}(R_{\bm{\lambda}} \, Z^{\tau_{\bm{\lambda}}}),$$
where
$Z^{\tau_{\bm{\lambda}}} =
\textup{sym}(\bm{\mathsf{d}}_{\tau_{\bm{\lambda}}}\bm{\mathsf{d}}_{\tau_{\bm{\lambda}}}^\top
)$ and $\bm{\mathsf{d}}_{\tau_{\bm{\lambda}}}$ is the vector containing all the
polynomials $\mathsf{d}^{\Theta_{\tau_{\bm{\lambda}}}}_F$ such that $F\in \mathcal{F}_T^{2d}$ where
$|T|=n-\lambda_1$.

For each intersection type $T$ with $|T|=n-\lambda_1$, let $\bm{\mathsf{d}}_{\tau_{\bm{\lambda}},T}$ be the
restriction of $\bm{\mathsf{d}}_{\tau_{\bm{\lambda}}}$ to entries corresponding to flags in $\mathcal{F}_{T}^{2d}$.  Similarly, let $R_{\bm{\lambda},T}$ be the
principal submatrix of $R_{\bm{\lambda}}$ corresponding to rows and columns indexed by flags in $\mathcal{F}_{T}^{2d}$. Since $R_{\bm{\lambda}}$
is psd, we get that $R_{\bm{\lambda},T}$ is psd for each $T$. From Lemma~\ref{lem:orthogonality},
we know that for any labeling $\Theta$,  if $F$ and $F'$ have different intersection types, then $\mathsf{d}_{F}^\Theta \mathsf{d}_{F'}^\Theta = 0$.
Hence, $Z^{\tau_{\bm{\lambda}}}_{F F'} = 0$ if $F$ and $F'$ have different intersection types, and so $Z^{\tau_{\bm{\lambda}}}$
is block diagonal, with one block for each intersection type. Thus, we have
\[ \mathsf{p} =
\sum_{\bm{\lambda}\in \Lambda} \sum_{\substack{T:\\ |T|=n-\lambda_1}}
\textup{tr}\left(R_{\bm{\lambda},T} \,\sym(\bm{\mathsf{d}}_{\tau_{\bm{\lambda}},T} \bm{\mathsf{d}}_{\tau_{\bm{\lambda}},T}^\top)\right).\]

Since $\lambda_1\geq n-2d$ for each partition $\bm{\lambda} \in \Lambda$ and $|T| = n - \lambda_1 \leq 2d$, by switching the sums, we may write the above expression as follows:

$$\mathsf{p} =  \sum_{\substack{T:\\0\leq |T| \leq 2d}} \sum_{\substack{\bm{\lambda}:\\ \lambda_1=n-|T|}}
\textup{tr}\left(R_{\bm{\lambda},T} \,\sym(\bm{\mathsf{d}}_{\tau_{\bm{\lambda}},T} \bm{\mathsf{d}}_{\tau_{\bm{\lambda}},T}^\top)\right).$$

Note that the vectors $ \bm{\mathsf{d}}_{\tau_{\bm{\lambda}},T}$ and $\bm{\mathsf{d}}_{\tau_{\bm{\lambda}'},T}$ are equal  if
$\textup{row}_1(\tau_{\bm{\lambda}}) = \textup{row}_1(\tau_{\bm{\lambda}'})$ since $\Theta_{\tau_{\bm{\lambda}}} =
\Theta_{\tau_{\bm{\lambda}'}}$. For every $\bm{\lambda}$ in the inner sum above, we could choose the first row of each
corresponding $\tau_{\bm{\lambda}}$ to be the same. Combining these two observations, we can reindex the vectors
$ \bm{\mathsf{d}}_{\tau_{\bm{\lambda}},T} $ by just $\bm{\mathsf{d}}_T$.
Therefore, the expression for $\mathsf{p}$ can be further rewritten as
$$\mathsf{p} =  \sum_{\substack{T:\\0\leq |T| \leq 2d}}
\textup{tr}\left(R_T \,\sym(\bm{\mathsf{d}}_T \bm{\mathsf{d}}_T^\top)\right) =
\sum_{t=0}^{2d} \sum_{T : |T| = t}
\textup{tr}\left(R_{T} \,\sym(\bm{\mathsf{d}}_{T} \bm{\mathsf{d}}_{T}^\top)\right),$$
where $R_T = \displaystyle{\sum_{\substack{\bm{\lambda}:\\ \lambda_1=n-|T|}}}   R_{\tau_{\bm{\lambda}},T}$ which is again psd.

\qed
\smallskip

The next step in getting to flag sos expressions is to switch from the operator $\textup{sym}(\cdot)$ to
the language of expectations. This  does not change the expressions we care about as we will see in
Lemma~\ref{lem:symvsexp}.

\begin{definition}
Fix $\Theta_0\in \textup{Inj}([t],[n])$, and $F,F' \in \mathcal{F}_T^f$ where $|T|=t$. We define
$$\mathbb{E}_{\Theta} [\mathsf{d}_F^{\Theta_0} \mathsf{d}_{F'}^{\Theta_0}] := \frac{1}{|\textup{Inj}([t],[n])|}\sum_{\Theta \in \textup{Inj}([t],[n])} \mathsf{d}^\Theta_F \mathsf{d}_{F'}^\Theta$$
\end{definition}

\begin{lemma}
Fix $\Theta_0\in \textup{Inj}([t],[n])$, and $F,F' \in \mathcal{F}_T^f$ where $|T|=t$, then
$$\textup{sym}(\mathsf{d}_F^{\Theta_0}\mathsf{d}_{F'}^{\Theta_0})=\frac{1}{|\mathfrak{S}_n|}\sum_{\mathfrak{s} \in \mathfrak{S}_n} \mathsf{d}_F^{\mathfrak{s}\cdot \Theta_0} \mathsf{d}_{F'}^{\mathfrak{s}\cdot \Theta_0},$$
where $\mathfrak{s}\cdot \Theta_0 (i) = \mathfrak{s}(\Theta_0(i))$.
\end{lemma}

\proof This follows from the definitions of the operator $\textup{sym}$ and polynomials $\mathsf{d}_F^{\Theta}$. \qed

\begin{lemma}\label{lem:symvsexp}
Fix $\Theta_0\in \textup{Inj}([t],[n])$, and $F,F' \in \mathcal{F}_T^f$ where $|T|=t$, then
$$\textup{sym}(\mathsf{d}_F^{\Theta_0} \mathsf{d}_{F'}^{\Theta_0}) = \mathbb{E}_\Theta [\mathsf{d}_F^{\Theta_0} \mathsf{d}_{F'}^{\Theta_0}]$$
\end{lemma}

\proof
This follows from the fact that $|\textup{Inj}([t],[n])|=\binom{n}{t}t!$ and that, for each $\Theta\in \textup{Inj}([t],[n])$, there are $(n-t)!$ permutations $\mathfrak{s}\in \mathfrak{S}_n$ such that $\mathfrak{s}\cdot \Theta_0=\Theta$.
\qed

\begin{definition} \label{def:flag sos}
Let $\bm{\mathsf{d}}^{\Theta,T,f}=(\mathsf{d}^\Theta_F)_{F\in \mathcal{F}_T^f}$ be the vector of flag polynomials for a  fixed intersection type $T$, flag size $f$, and labeling $\Theta$. A {\em flag sos} is a sos expression of the form
\begin{align*} \label{eq:raz_sos}
\sum_{T,f} \textup{tr}\left(R_{T,f} \, \mathbb{E}_{\Theta}  \left[\bm{\mathsf{d}}^{\Theta,T,f}{\bm{\mathsf{d}}^{\Theta,T,f}}^\top\right]\right),
\end{align*}
where the sum is indexed by some chosen $T$ and $f$, and the matrices $R_{T,f}$ are psd. This expression is called  {\em $f_{\textup{max}}$-flag sos} if every flag present has size at most $f_{\textup{max}}$.
\end{definition}

Note that the main difference between the sos expression in Corollary~\ref{cor:dsumsofsquares} and a flag sos expression is that each summand in the latter uses only one intersection type.  However, we saw in Proposition~\ref{prop:new dsumsofsquares} that we can refine the sos in Corollary~\ref{cor:dsumsofsquares} to have the same property.  The only step left is to bring in the language of expectations to obtain flag sos expressions.

\begin{theorem} \label{thm:GP=Razborov}
If $\mathsf{p}$ is symmetric and $d$-sos, then $\mathsf{p}$ is also a $2d$-flag sos.
\end{theorem}

\proof
By Proposition~\ref{prop:new dsumsofsquares}, if $\mathsf{p}$ is symmetric and $d$-sos, then
$$\mathsf{p} =  \sum_{t=0}^{2d} \sum_{T : |T| = t}
\textup{tr}\left(R_{T} \,\sym(\bm{\mathsf{d}}_{T} \bm{\mathsf{d}}_{T}^\top)\right),$$
where $\bm{\mathsf{d}}_T$ is the vector of polynomials
$\bm{\mathsf{d}}_F^{\Theta_T}$ such that $F \in \mathcal{F}_{T}^{2d}$ and $\Theta_T \in \textup{Inj}([t],[n])$. By Lemma~\ref{lem:symvsexp}, we then have that
$$\mathsf{p} =  \sum_{t=0}^{2d} \sum_{T : |T| = t}
\textup{tr}\left(R_T\mathbb{E}_{\Theta}[\bm{\mathsf{d}}_{T} \bm{\mathsf{d}}_{T}^\top]\right).$$
\qed

In \cite{RST}, we proved that any flag sos polynomial of degree $2d$ can be written as a sos coming from $\mathbb{R}[\mathsf{x}]_{\leq d}$ using the techniques of \cite{GatermannParrilo}. We have now shown that the reverse is also true in the sense that any symmetric sos over $\mathcal{V}_n$ can also be written as a flag sos. Thus flag methods are equivalent to general methods for finding sos expressions over $\mathcal{V}_n$.

% !TEX root =  main.tex
\section{Examples} \label{sec:examples}

In this section, we illustrate our methods on two examples. 
The first is a sos proof of an upper bound on the edge density of a $n$-vertex graph that does not contain any $4$-cycles, and the second is a sos proof that the Ramsey number $R(3,3)=6$. These example illustrates various features of our method. The first is
that we can work with additional constraints beyond those defining the discrete
hypercube.  The second is that in both examples, $n$ is fixed, and hence they showcase the merits of
considering flags in a finite setting. Thirdly, the polynomials $\mathsf{g}_F^\Theta$ and $\mathsf{d}_F^\Theta$ that appear in these proofs are inspired by the combinatorics of the problems at hand. Normally, the bases used in symmetry-reduction, 
such as the Gatermann-Parrilo method, does not preserve this sort of combinatorial information. 

\subsection{Avoiding $4$-cycles}\label{sec:4cycle}
In \cite{KST54}, K\H ovari, S\'os and Tur\'an proved that the number of edges in a $n$-vertex graph not containing $C_4$, the cycle on four vertices, is bounded above by $\frac12 n^{3/2}+O(n)$. Observe that the bound implies that the edge density of an extremal graph in this setting tends to zero, and thus the standard application of the flag algebra method cannot recover the precise lower order term. We give a succinct sos proof for this result using our methods which follows the proof outline in \cite{KST54}.

\begin{theorem}\cite{KST54}
Let $G$ be a $n$ vertex graph not containing a $C_4$. Then the number of edges in $G$ is at most $\frac{1}{2} n^{\frac{3}{2}}+ O(n)$.
\end{theorem}
\begin{proof}
Fix $n$ and make the following definitions:
\begin{align*}
\mathsf{s}&:=  \sum_{1\leq i <j\leq n} \mathsf{x}_{ij}, \textup{ and }\\
\mathcal{I}&:=\langle \mathsf{x}_{ij}^2-\mathsf{x}_{ij} \, \forall 1\leq i< j\leq n, \, \mathsf{x}_{ij}\mathsf{x}_{ik}\mathsf{x}_{lj}\mathsf{x}_{lk} \, \forall i,j,k,l\in [n]\rangle.
\end{align*}
Note that the variety of $\mathcal{I}$ is the subset of $\mathcal{V}_n$ that consists of characteristic vectors of $n$-vertex graphs that avoid a $C_4$. We will show that
$n + \frac{2}{n-1}\mathsf{s} - \frac{2}{\binom{n}{2}} \mathsf{s}^2$ is $2$-sos modulo the ideal $\mathcal{I}$. This will imply that $\mathsf{s} \leq \frac{n+\sqrt{4n^3-3n^2}}{4}$ for all $C_4$-free graphs on $n$ vertices
giving the bound claimed in the theorem.

We first consider the spanning set given by the $\mathsf{g}_F^\Theta$ polynomials.
Consider flags of size three with the intersection type $\labeledvnonedge{1}{2}$, and let $\Theta_{jk}\in \textup{Inj}([2],[n])$ be such that $\Theta(1)=j$ and $\Theta(2)=k$. Then, 

\begin{align*}
\mathsf{g}^{\Theta_{jk}}_{\labeledHzero{}{1}{2}}=n-2,\qquad \qquad  &\mathsf{g}^{\Theta_{jk}}_{\labeledHone{}{1}{2}}= \sum_{i \in [n]\backslash\{j,k\}} \mathsf{x}_{ij},\\
 \mathsf{g}^{\Theta_{jk}}_{\labeledHoneright{}{1}{2}}= \sum_{i\in [n]\backslash\{j,k\}} \mathsf{x}_{ik}, \qquad \qquad & \mathsf{g}^{\Theta_{jk}}_{\labeledHtwo{}{1}{2}}= \sum_{i\in [n]\backslash\{j,k\}} \mathsf{x}_{ij}\mathsf{x}_{ik}.
\end{align*}
We now claim that the following sos certificate is correct and thus prove the theorem.
\begin{align*}
n + \frac{2}{n-1}\mathsf{s} &- \frac{2}{\binom{n}{2}} \mathsf{s}^2\\
&\equiv \mathbb{E}_{\Theta_{jk}} \left[ n\left(\frac{1}{n-2} \mathsf{g}^{\Theta_{jk}}_{\labeledHzero{}{1}{2}}-\mathsf{g}^{\Theta_{jk}}_{\labeledHtwo{}{1}{2}}\right)^2 +\frac{1}{2}\left(\mathsf{g}^{\Theta_{jk}}_{\labeledHone{}{1}{2}}-\mathsf{g}^{\Theta_{jk}}_{\labeledHoneright{}{1}{2}}\right)^2\right] \textrm{  mod }\mathcal{I}.
\end{align*}

The above expression follows from the following equalities.
First, for each $j\neq k$, we have the following equivalence:
\begin{align*}
&\left(\frac{1}{n-2} \mathsf{g}^{\Theta_{jk}}_{\labeledHzero{}{1}{2}}-\mathsf{g}^{\Theta_{jk}}_{\labeledHtwo{}{1}{2}}\right)^2\\
& \,\,\,\,\,\,\,\,\,= \left(1-\sum_{i\in [n]\backslash\{j,k\}} \mathsf{x}_{ij}\mathsf{x}_{ik}\right)^2 \\
&\,\,\,\,\,\,\,\,\, =1+\sum_{i\in [n]\backslash\{j,k\}} \mathsf{x}_{ij}^2\mathsf{x}_{ik}^2 -2 \sum_{i\in [n]\backslash\{j,k\}} \mathsf{x}_{ij}\mathsf{x}_{ik} + \sum_{i\neq l:i,l\in [n]\backslash\{j,k\}} \mathsf{x}_{ij}\mathsf{x}_{ik}\mathsf{x}_{lj}\mathsf{x}_{lk}\\
&\,\,\,\,\,\,\,\,\, \equiv \left(1-\sum_{i\in [n]\backslash\{j,k\}} \mathsf{x}_{ij}\mathsf{x}_{ik}\right) \textrm{ mod } \mathcal{I}.\\
\end{align*}
Thus,
\begin{align*}
\mathbb{E}_{\Theta_{jk}} &\left[ n\left(\frac{1}{n-2} \mathsf{g}^{\Theta_{jk}}_{\labeledHzero{}{1}{2}}-\mathsf{g}^{\Theta_{jk}}_{\labeledHtwo{}{1}{2}}\right)^2 \right] 
\equiv n-\frac{1}{n-1}\sum_{j,k\in [n],j\neq k}\sum_{i\in [n]\backslash\{j,k\}} \mathsf{x}_{ij}\mathsf{x}_{ik}\textrm{ mod } \mathcal{I}.
\end{align*}

Next, 
\begin{align*}
&\mathbb{E}_{\Theta_{jk}} \left[\frac{1}{2}\left(\mathsf{g}^{\Theta_{jk}}_{\labeledHone{}{1}{2}}-\mathsf{g}^{\Theta_{jk}}_{\labeledHoneright{}{1}{2}}\right)^2\right]\\
&\,\,\,\,\,\,\,\,\,=\frac{1}{n(n-1)} \sum_{\substack{j,k\in [n] :\\ j\neq k}} \frac{1}{2} \left(\sum_{i\in [n]\backslash\{j\}} \mathsf{x}_{ij}-\sum_{i\in [n]\backslash\{k\}}\mathsf{x}_{ik}\right)^2\\
&\,\,\,\,\,\,\,\,\,=\frac{1}{n(n-1)} \left( n \sum_{j\in [n]} \left(\sum_{i\in [n]\backslash\{j\}} \mathsf{x}_{ij}\right)^2- \left(\sum_{\substack{i,j\in [n] :\\ i\neq j}} \mathsf{x}_{ij}\right)^2\right)\\
&\,\,\,\,\,\,\,\,\,=\frac{1}{n-1}\sum_{\substack{j,k\in [n]:\\j\neq k}} \left(\sum_{i\in [n]\backslash\{j,k\}} \mathsf{x}_{ij}\mathsf{x}_{ik} \right)+ \frac{1}{n-1}2\mathsf{s}-\frac{1}{n(n-1)} 4\mathsf{s}^2.
\end{align*}

In order to check the above equations, it helps to think of $\sum_{i\in [n]\backslash\{j\}} \mathsf{x}_{ij}$ for a fixed $j$ as the degree of vertex $j$, and $\sum_{i,j\in [n] : i\neq j} \mathsf{x}_{ij}$ as the sum of the degrees of all vertices. Combining the two equalities above, yields the desired sos expression.

Similarly, for the $\mathsf{d}_F^\Theta$'s, we again consider flags of size three with the intersection types $\labeledvnonedge{1}{2}$ and $\labeledvedge{1}{2}$. We also again let $\Theta_{jk}\in \textup{Inj}([2],[n])$ be such that $\Theta(1)=j$ and $\Theta(2)=k$. This gives us the following sos expression
\begin{align*}
n + \frac{2}{n-1} \mathsf{s} &- \frac{2}{\binom{n}{2}} \mathsf{s}^2\equiv\\
\mathbb{E}_{\Theta_{jk}} \Bigg[& {n}\left(\mathsf{d}^{\Theta_{jk}}_{\labeledHzero{}{1}{2}} + \mathsf{d}^{\Theta_{jk}}_{\labeledHone{}{1}{2}} + \mathsf{d}^{\Theta_{jk}}_{\labeledHoneright{}{1}{2}}\right)^2 + {n}\left(\mathsf{d}^{\Theta_{jk}}_{\labeledHonebottom{}{1}{2}} +\mathsf{d}^{\Theta_{jk}}_{\labeledHtwoleft{}{1}{2}}+\mathsf{d}^{\Theta_{jk}}_{\labeledHtworight{}{1}{2}}\right)^2\\
& + \frac{1}{2}\left(\mathsf{d}^{\Theta_{jk}}_{\labeledHone{}{1}{2}}-\mathsf{d}^{\Theta_{jk}}_{\labeledHoneright{}{1}{2}}\right)^2 + \frac{1}{2}\left(\mathsf{d}^{\Theta_{jk}}_{\labeledHtwoleft{}{1}{2}}-\mathsf{d}^{\Theta_{jk}}_{\labeledHtworight{}{1}{2}}\right)^2 \Bigg] \textrm{ mod } \mathcal{I},
\end{align*}
where each summand involves a single intersection type. 

\end{proof}

\subsection{A Ramsey example}\label{sec:ramsey}

Next we give a sos proof that the Ramsey number
$R(3,3)=6$.  The Ramsey $R(r,s)$ problem consists of finding the minimum number of vertices $n$ such that it is impossible to color all the edges of the complete graph $K_n$ blue and red without creating either a blue clique of size $r$ or a red clique of size $s$~\cite{Radziszowski94}. Ramsey's theorem states that this number is finite for any $r,s\in \mathbb{Z}$. Several values of $R(r,s)$ are known. For example, $R(3,3)=6$: no matter how one colors the edges of $K_6$, one is guaranteed to have a monochromatic triangle.

We consider the following model. We let $\mathsf{x}_e$ be one if edge $e$ is colored red and zero if edge $e$ is colored blue; $\mathsf{x}\in \mathbb{R}^{\binom{6}{2}}$. Consider the ideal
\begin{align*}
\mathcal{I} = \langle \mathsf{x}_e^2-\mathsf{x}_e \,\, \forall  e\in E(K_6) \rangle & +  \langle \mathsf{x}_{ij}\mathsf{x}_{ik}\mathsf{x}_{jk} \,\,\forall i<j<k\in [6] \rangle\\
 & +  \langle (1-\mathsf{x}_{ij})(1-\mathsf{x}_{ik})(1-\mathsf{x}_{jk}) \,\, \forall i<j<k\in [6] \rangle.
\end{align*}
The variety of $\mathcal{I}$ consists of those vectors in $\{0,1\}^{6 \choose 2}$ that record a red/blue coloring of $K_6$ without monochromatic triangles since the first part ensures that each
$x_e$ is either $0$ or $1$, the second part, that no triangle is red, and the third part, that no triangle is blue.
Showing that the variety of $\mathcal{I}$ is empty will prove that $R(3,3) \leq 6$. The coloring of $K_5$ below establishes that $R(3,3) > 5$.

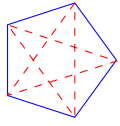
\newdimen\R
\R=0.8cm

\begin{center}\begin{tikzpicture}
        \draw[xshift=5.0\R, blue] (0:\R) \foreach \x in {72,144,...,359} {
            -- (\x:\R)
        } -- cycle (90:\R)  ;
   \draw[xshift=5.0\R, red, dashed] (72:\R) \foreach \x in {216,359,144,288,72} {
            -- (\x:\R)
        } ;
\end{tikzpicture}\end{center}

To prove that the variety of $\mathcal{I}$ is empty, it suffices to show that $-1$ is equivalent to a sos mod $\mathcal{I}$. We now provide sos expressions that do this job in our setting.

Consider the $\mathsf{d}_F^\Theta$'s from the the following flags: $\vedge$, $\vnonedge$, $\labeledvedge{}{1}$ and $\labeledvnonedge{}{1}$. Note that for the first two, $\Theta$ is irrelevant and unique since the intersection type is empty. For the last two, we assume that $\Theta_i$ sends $1$ to vertex $i$. We have

$$\mathsf{d}^\Theta_{\vedge}=2\sum_{1\leq i<j \leq 6} \mathsf{x}_{ij}, \quad \quad \mathsf{d}^\Theta_{\vnonedge}=2\sum_{1\leq i < j \leq 6} (1-\mathsf{x}_{ij}),$$
$$\mathsf{d}^{\Theta_i}_{\labeledvedge{}{1}}=\sum_{j\in [6]\backslash\{i\}} \mathsf{x}_{ij}, \quad \quad \mathsf{d}^{\Theta_i}_{\labeledvnonedge{}{1}}=\sum_{j\in[6]\backslash\{i\}}^6 (1-\mathsf{x}_{ij}).$$

Similarly, for the $\mathsf{g}_F^\Theta$'s, we need the flags: $\vnonedge$, $\labeledvedge{}{1}$ and $\labeledvnonedge{}{1}$. Note that for the first one, $\Theta$ is irrelevant and unique since the intersection type is empty. For the last two, we assume that $\Theta_i$ sends $1$ to vertex $i$. We have

$$\mathsf{g}^\Theta_{\vnonedge}=30, \quad \quad \mathsf{g}^{\Theta_i}_{\labeledvedge{}{1}}=\sum_{j\in [6]\backslash\{i\}} \mathsf{x}_{ij}, \quad \quad \mathsf{g}^{\Theta_i}_{\labeledvnonedge{}{1}}=5.$$

\begin{theorem} \label{thm:d,g, proof of Ramsey}
The following expressions
$$\frac{1}{8\binom{6}{2}^2} \left( \mathsf{d}^\Theta_{\vedge} + \mathsf{d}^\Theta_{\vnonedge}\right)^2 + \mathbb{E}_{\Theta_i}\left[\frac{1}{2}\left(\mathsf{d}^{\Theta_i}_{\labeledvedge{}{\textup{1}}}-\mathsf{d}^{\Theta_i}_{\labeledvnonedge{}{\textup{1}}} \right)^2\right] $$
$$\textup{ and } \frac{1}{2}\left(\frac{1}{30} \mathsf{g}^\Theta_{\vnonedge}\right)^2 + \mathbb{E}_{\Theta_i}\left[\left(\sqrt{2}\mathsf{g}^{\Theta_i}_{\labeledvedge{}{\textup{1}}}-\frac{1}{\sqrt{2}}\mathsf{g}^{\Theta_i}_{\labeledvnonedge{}{\textup{1}}}\right)^2  \right]$$
are each equivalent to $-1$ modulo $\mathcal{I}$ and hence the variety of $\mathcal{I}$ is empty.
\end{theorem}

\proof
Check that  the expressions in the $\mathsf{d}^\Theta_F$'s and $\mathsf{g}^\Theta_F$'s are each equal to
$$\textup{sym}\left( \left(\frac{1}{\sqrt{2}}\right)^2 + \left(\sqrt{2} \sum_{i=2}^6 x_{1i} - \frac{5}{\sqrt{2}}\right)^2\right).$$
This sos is equal to another sos
$$\textup{sym}\left(\left(2-\sum_{i=2}^6 x_{1i}\right)^2+\left(2-\sum_{i=2}^6 \left(1-x_{1i}\right)\right)^2\right),$$
which is equivalent to $-1$ modulo $\mathcal{I}$ by Lemma~\ref{lem:sos proof of Ramsey}.
\qed

% !TEX root =  main.tex
\section{An example needing infinitely many flags}  \label{sec:infinite}
\label{sec:infinite_example}
	For every positive integer $n$, define the degree two polynomial
	\begin{equation}
	\label{eq:inf-fn} 
		\mathsf{f}_n=\frac{1}{\binom{n}{2}^2}\left(\sum_{e\in E(K_n)} \mathsf{x}_e - \left\lfloor \frac{\binom{n}{2}}{2} \right\rfloor\right)
		\left(\sum_{e\in E(K_n)} \mathsf{x}_e - \left\lfloor \frac{\binom{n}{2}}{2} \right\rfloor - 1\right).
	\end{equation}
	Observe that $\mathsf{f}_n$ satisfies $0\leq \mathsf{f}_n(\mathsf{x})\leq 1$ for all $\mathsf{x}\in \cV_n$. 
	(Here the upper bound is not tight.)
	In this section, we show that there is no finite collection of flags $\mathcal{F}$ 
	such that each $\mathsf{f}_n$ is equivalent to a sos (modulo $\mathcal{I}_n$) of flag polynomials $d_{\Theta}^F$ for $F\in \mathcal{F}$.

	Despite being of degree two and non-negative on $\cV_n$,
	Grigoriev~\cite{grigoriev2001complexity} showed that any sos $\mathsf{h}_n$, such that
	$\mathsf{h}_n(\mathsf{x}) = \mathsf{f}_n(\mathsf{x})$ for all $\mathsf{x}\in\cV_n$, necessarily has degree
	$\Omega\left(\binom{n}{2}\right)$ (see also~\cite{BlekhermanGouveiaPfeiffer}
	for an extension to the case in which $h$ is rational).  Recently, Lee,
	Prakesh, de Wolf, and Yuen~\cite{lee2016sum} extended this degree lower bound to
	the case in which $\mathsf{h}_n$ approximates $\mathsf{f}_n$ in the
	$\ell_{\infty}$ sense. The following is a special case of~\cite[Theorem 1.1]{lee2016sum}.
\begin{theorem}
	\label{thm:lee}
	There exists a (sufficiently small) positive constant $c$ such that for any positive integer $n$, 
	if $\mathsf{h}_n$ is $d$-sos and satisfies
	\[|\mathsf{h}_n(\mathsf{x})-\mathsf{f}_{n}(\mathsf{x})| \leq 
		\frac{(1/50)}{\binom{n}{2}^2}\]
	for all $\mathsf{x}\in \cV_n$, then $d\geq c\binom{n}{2}$.
\end{theorem}
Since the minimum value of $\mathsf{f}_n$ on $\RR^{\binom{n}{2}}$ is
$-(1/4)/\binom{n}{2}^2$, it follows that $\mathsf{f}_n(\mathsf{x})+(1/4)/\binom{n}{2}^2$ is
a non-negative quadratic on $\RR^{\binom{n}{2}}$, and so it is a sos of degree one polynomials. As such, we certainly cannot hope to tolerate errors
larger than $(1/4)/\binom{n}{2}^2$ in Theorem~\ref{thm:lee}.

Any symmetric sos of flag polynomials $\mathsf{d}^{\Theta}_F$ using
a fixed collection of flags $F$ necessarily has bounded degree.  Hence the
(growing) lower bound on degree from Theorem~\ref{thm:lee} implies that sums of
squares of flag polynomials using a fixed collection of flags cannot be used to 
certify non-negativity of every $\mathsf{f}_n$, $n\geq 1$.
\begin{proposition}
	\label{prop:inf-flag}
	For any fixed positive integer $f$, there does not exist a $f$-flag sos $\mathsf{p}$ such that 
	 $|\mathsf{p}(\mathsf{x})-\mathsf{f}_n(\mathsf{x})| \leq \frac{(1/50)}{\binom{n}{2}^2}$ 
	for all $\mathsf{x}\in\cV_n$ and all sufficiently large $n$. 
\end{proposition}
\begin{proof}
	If $\mathsf{p}$ is a $f$-flag sos, then $\mathsf{p}$ is $\binom{f}{2}$-sos. If $n$ is large enough, 
	then $c\binom{n}{2} > \binom{f}{2}$, where $c$ is the constant in Theorem~\ref{thm:lee}. 	
\end{proof}

This shows that there is a sequence of symmetric polynomials of degree two, all
of which are non-negative and bounded on $\cV_n$, that cannot
be approximated within $O(1/n^4)$ by $f$-flag sums of squares for any fixed $f$.
We make several comments. 

Razborov's application of flag density polynomials typically allows much larger
errors of size $O(1/n)$, suggesting the following question. Is there a sequence
of symmetric polynomials taking values in $[0,1]$ that cannot be 
approximated within $O(1/n)$ by sums of squares from a finite collection of flags?

In \cite{RazborovFlagAlgebras}, Razborov asked whether every asymptotic graph density inequality can be certified 
by a finite set of flags. As an answer to this, Hatami and Norine proved that deciding the non-negativity of 
a linear inequality in graph densities is in general undecidable \cite{Hatami-Norine}. 
Proposition~\ref{prop:inf-flag} can be viewed as an answer to Razborov's question in our setting of the discrete hypercube 
for finite $n$. 

At a quick glance, Proposition~\ref{prop:inf-flag} may seem to contradict our earlier result that 
it is possible to write a flag sos expression that depends only on the degree of a 
given sos expression and independent of $n$.  For the family of quadratic 
polynomials $\mathsf{f}_n$, the degree of any sos expression is a function of 
$n$ and therefore, by our results, the number of flags needed in a flag sos 
expression also depends on $n$.

% !TEX root =  main.tex

\section{Summary and Discussion}  \label{sec:discussion}

Our results give the following framework for searching for a succinct sos expression for a symmetric polynomial $\mathsf{p}$ that is $d$-sos.

\begin{enumerate}

\item For each partition $\bm{\lambda} \in \Lambda = \{ \bm{\lambda} \vdash n \,:\, {\bm{\lambda}} \geq_{\textup{lex}} (n-2d,1^{2d})\}$, fix a standard tableau $\tau_{\bm{\lambda}}$ of shape $\bm{\lambda}$.

 \item Pick polynomials  $\mathsf{p}_1^{\tau_{\bm{\lambda}}}, \ldots, \mathsf{p}_{l_{\tau_{\bm{\lambda}}}}^{\tau_{\bm{\lambda}}}$ whose span contains $W_{\tau_{\bm{\lambda}}}$. 
 
\item Introduce a matrix of variables $R_{\bm{\lambda}}$ of size $l_{\tau_{\bm{\lambda}}} \times l_{\tau_{\bm{\lambda}}}$. 

\item Formulate a SDP using the following polynomial equivalence
 $$ \mathsf{p} = \sum_{\bm{\lambda} \in \Lambda} \textup{tr}( R_{\bm{\lambda}}\, Z^{\tau_{\bm{\lambda}}})$$
 where $Z^{\tau_{\bm{\lambda}}}_{ij} := \textup{sym}(\mathsf{p}_i^{\tau_{\bm{\lambda}}} \mathsf{p}_j^{\tau_{\bm{\lambda}}})$.
\end{enumerate}

We summarize the computational savings in each of the above steps.
First of, the symmetry of $\mathsf{p}$ allows us to work with a \emph{single} standard tableau $\tau_{\bm{\lambda}}$ of shape $\bm{\lambda}\in \Lambda$ as in Corollary~\ref{cor:GP with big lambda}. 
By Proposition~\ref{prop:total number of m lambdas}, the size of $\Lambda$, and hence the number of partitions that are needed for sos expressions in our method, is independent of $n$. 
Our framework gives us the flexibility to work with a spanning set for $W_{\tau_{\bm{\lambda}}}$ as opposed to a basis. 
 We constructed spanning sets in Section~\ref{sec:spanning sets} whose 
cardinalities are independent of $n$ as proved in Proposition~\ref{prop:mlambda-indn}.
The size of the formulated SDP depends on the total sum of dimensions $m_{\bm{\lambda}}$ of $W_{\tau_{\bm{\lambda}}}$ for $\bm{\lambda} \in \Lambda$.  Combining 
Propositions~\ref{prop:total number of m lambdas} and ~\ref{prop:mlambda-indn}, we get that  
$\sum_{{\bm{\lambda}} \in \Lambda} m_{\bm{\lambda}}$ is also independent of $n$. 

Our spanning sets for $W_{\tau_{\bm{\lambda}}}$ can have combinatorial meaning and structure 
relevant to the polynomial $\mathsf{p}$ in some instances.  For example, in extremal combinatorics, the density polynomials from flag algebras have simple combinatorial interpretations as compared to the basis necessary for the Gatermann-Parrilo method. We also saw in Section~\ref{sec:spanning sets} that the concept of flags arise naturally in our theory from the action of $\mathfrak{S}_n$ on square-free monomials. 

We now make several comments on various features, connections and extensions of our method. 
While we presented our results for the hypercube $\mathcal{V}_n$, the arguments naturally generalize to subsets of $\mathcal{V}_n$ by working with ideals that contain 
$\mathcal{I}_n$ and are invariant under $\mathfrak{S}_n$. There are several ways to approach this general situation. If the ideal in question is $\mathcal{I}$, 
then for instance, one could look for symmetry-reduced sos expressions from the vector space $V = \RR[\cV_n]_{\leq d}$
as we have been doing in this paper. In this case, our methods work as is and we can further reduce the obtained 
sos expression by the ideal $\mathcal{I}$ as seen in the Ramsey example in Section~\ref{sec:ramsey}. Alternately, one could work with $V = \mathbb{R}[{\mathsf{x}}]_{\leq d} / \mathcal{I}$ or $V =  \left( \mathbb{R}[{\mathsf{x}}] / \mathcal{I} \right)_{\leq d}$. In both these cases, our strategy would work in principle but we would need to understand the isotypic decomposition of $V$ and the associated vector spaces $W_{\tau_{\bm{\lambda}}}$.

In contrast to the flag algebra  and graph homomorphism frameworks, our methods work on graphs of finite size when applied to graph density problems. This allows us to obtain new types of exact sos proofs that apply to finite, as opposed to asymptotic situations. 
For example, in Section~\ref{sec:ramsey}, we gave a sos proof for an upper bound on the Ramsey number $R(3,3)$ in which case $n=6$. 
Our methods also naturally give upper bounds on extremal graph theoretic problems in the sparse regime. For example, we saw 
how to use our method to obtain a succinct sos proof of the result of K\H ovari, S\'os and Tur\'an~\cite{KST54} that a $n$-vertex graph not containing $C_4$ has at most $\frac{1}{2} n^{3/2} +O(n)$ edges. Since the extremal graphs are not dense in this case, applying the flag algebra framework gives an upper bound of zero on the limiting edge density but our method can give a precise estimate of the lower order terms. 

Finally, as promised in the Introduction, our results extend naturally to the hypercube $\mathcal{V}_{n,k} := \{0,1\}^{n \choose k}$. 
We would use the polynomial ring with variables $\mathsf{x}_{i_1\ldots i_k}$ indexed by the edges in the complete $k$-uniform hypergraph $K_n^k$ with $\mathfrak{S}_n$ acting on monomials via $\mathfrak{s} \cdot \mathsf{x}_{i_1\ldots i_k}:=\mathsf{x}_{\mathfrak{s}(i_1)\ldots \mathfrak{s}(i_k)}$ for each $\mathfrak{s}\in \mathfrak{S}_n$. The ideal needed here is 
$\mathcal{I}_{n,k} :=\langle \mathsf{x}_{i_1 \ldots i_k}^2 - \mathsf{x}_{i_1 \ldots i_k} \, \forall \{i_1,\ldots, i_k\}\in \binom{[n]}{k} \rangle$. All of our results and their proofs can be translated in a straightforward manner to this setting and we give a few samples below. 

\begin{itemize}

\item In  Theorem \ref{thm:isodecomp}, the multiplicity $m_{\bm{\lambda}}$ of $S_{\bm{\lambda}}$ in the decomposition of $V=\mathbb{R}[{\mathsf{x}_{i_1\ldots i_k}}]_{\leq d}$ into irreducible $\fS_n$-modules is zero unless ${\bm{\lambda}}\geq_{\textup{lex}}(n-kd,1^{kd})$, i.e., $V = \bigoplus_{{\bm{\lambda}} \geq_{\textup{lex}} (n-kd,1^{kd})} V_{\bm{\lambda}}.$ Thus the number of partitions with non-zero $m_{\bm{\lambda}}$ can be bounded by $p(0)+p(1)+\ldots+p(kd)$ generalizing Proposition~\ref{prop:total number of m lambdas}.
    
\item In Theorem \ref{thm:span}, the polynomials $\mathsf{g}_F^\Theta$ (respectively $\mathsf{d}_F^\Theta$) with flags $F\in \mathcal{F}_T^{kd}$ (with 
the same restrictions) span $W_{\tau_{\bm{\lambda}}}$. 
    The sizes of these spanning sets will be independent of $n$ as in Proposition~\ref{prop:mlambda-indn}.

        \item In Theorem \ref{thm:GP=Razborov},  if a non-negative symmetric polynomial $\mathsf{p}\in \mathbb{R}[\mathcal{V}_{n,k}]$ is $d$-sos, then $\mathsf{p}$ also has a flag sos certificate coming from flags with at most $kd$ vertices.
\end{itemize}

\bibliography{RSST}
\bibliographystyle{alpha}
\addresseshere

\begin{appendix}
\newpage
\section{Proof of Theorem~\ref{thm:GP}}
\label{app:appendix1}

The aim of this appendix is to prove Theorem~\ref{thm:GP}, describing the structure of
symmetry-reduced certificates of non-negativity for $\fS_n$-invariant polynomials
that are $d$-sos. Rather than proving Theorem~\ref{thm:GP} directly, we establish a slightly more general result (Theorem~\ref{thm:GP-app} to follow)
and then specialize to give Theorem~\ref{thm:GP}.
This more general result is essentially the main result of Gatermann and Parrilo~\cite{GatermannParrilo}.
In Section~\ref{sec:app-bg}, we summarize some basic facts about real representations needed for the appendix.
In Section~\ref{sec:app-symsq}, we examine what happens when we symmetrize products and squares of polynomials under
a finite group action, leading to a proof of Theorem~\ref{thm:GP-app} and of Theorem~\ref{thm:GP}.

\subsection{Preliminaries on real representations}
\label{sec:app-bg}
Let $G$ be a finite group and let $U$ be a finite-dimensional $\RR G$-module, (i.e., a real vector space with an action of $G$ by linear transformations).
An $\RR G$-module $U$ is \emph{irreducible} if
the only real vector subspaces of $U$ invariant under the action of $G$ are $\{0\}$ and $U$.
Given any pair $U,W$ of $\RR G$-modules, let $\textup{Hom}_G(U,W)$ be the vector space of $\RR$-linear maps
$\phi:U\rightarrow W$ such that $\phi(g\cdot u) = g\cdot \phi(u)$ for all $u\in U$.
If $U,W$ are $\RR G$-modules, then the tensor product (over $\RR$), $U\otimes W$, is a $\RR G$-module
with the action $g\cdot (u\otimes w) = (g\cdot u)\otimes (g\cdot w)$.
If $U$ is a $\RR G$-module, then the space $U^*$ of $\RR$-valued linear
functionals on $U$ is a $\RR G$-module via $(g\cdot \ell)(u) = \ell(g^{-1}\cdot u)$.

If $U$ is a $\RR G$-module, define the linear map $\sym_{G}:U\rightarrow U$
by
\[ \sym_G(u) = \frac{1}{|G|}\sum_{g\in G} g\cdot u.\]
Clearly, the image of $\sym_G$ is the set of elements of $U$ fixed by the action of $G$, i.e.,
$U^G:=\{u\in U:\; g\cdot u = u\;\;\textup{for all $g\in G$}\}$. We omit the subscript $G$ from $\sym_G$ if
the group is clear from the context. We also note that $\sym_{G}\in \textup{Hom}_G(U,U^G)$ where $U^G$ carries
the trivial action of $G$.

We now recall the the isotypic decomposition of a $\RR G$-module $V$.
Let $(S_{\lambda})_{\lambda\in \Lambda}$ be an enumeration of inequivalent
irreducible $\RR G$-modules and let $n_\lambda = \dim(S_\lambda)$.
For each $\lambda \in \Lambda$, let
$\textup{Hom}_{G}(S_\lambda,V)$ denote the $m_\lambda$-dimensional \emph{multiplicity space} of $S_\lambda$ in $V$.
Let
\[ V_{\lambda} = \textup{span}\{\phi(s): \phi\in \textup{Hom}_G(S_\lambda,V),\;\;s\in S_{\lambda}\}\]
denote the \emph{isotypic component} of $V$ corresponding to $S_{\lambda}$. The decomposition
$V = \bigoplus_{\lambda\in \Lambda} V_{\lambda}$ is the \emph{isotypic decomposition} of $V$.

There are two natural ways to construct subspaces of the isotypic components $V_{\lambda}$. If we fix  a non-zero $\phi\in \textup{Hom}_G(S_\lambda,V)$
and consider $\{\phi(s): s\in S_\lambda\}$, we obtain a $G$-invariant subspace of $V_{\lambda}$ that is isomorphic to $S_{\lambda}$ and
has dimension $n_\lambda$. On the other hand, if we fix a non-zero $s_\lambda\in S_{\lambda}$,
we can define
\begin{equation}
	\label{eq:Ws}
	 W_{s_\lambda} := \{\phi(s_\lambda): \phi\in \textup{Hom}_G(S_\lambda,V)\}\subseteq V_\lambda.
\end{equation}
 This is a vector subspace
of $V_{\lambda}$ of dimension $m_\lambda$, but is not $G$-invariant. It is however isomorphic
to the multiplicitly space $\textup{Hom}_G(S_\lambda,V)$ as a vector space.
\begin{lemma}
	\label{lem:mult-iso}
	Fix some non-zero $s_\lambda \in S_\lambda$. Then the map $\Phi_{s_\lambda}: \textup{Hom}_G(S_\lambda,V)\rightarrow W_{s_{\lambda}}$
	defined by $\Phi_{s_\lambda}(\phi) = \phi(s_\lambda)$ is an isomorphism of vector spaces.
\end{lemma}
\begin{proof}
	The map $\Phi_{s_\lambda}$ is clearly linear. It is surjective by the definition of $W_{s_\lambda}$.
	It remains to show that $\Phi_{s_\lambda}$ is injective. Let $\phi \in \textup{Hom}_G(S_\lambda,V)$ and suppose that
	$\Phi_{s_\lambda}(\phi) = \phi(s_\lambda) = 0$.
	Since $\phi\in \textup{Hom}_G(S_\lambda,V)$, the kernel of $\phi$ is an invariant subspace of $S_\lambda$ that contains the non-zero element
	$s_\lambda$. Since $S_\lambda$ is irreducible, the kernel of $\phi$ must then be all of $S_\lambda$, and so $\phi=0$. Hence $\Phi_{s_\lambda}$
	is injective and so an isomorphism of vector spaces.
\end{proof}

The following standard fact about invariant bilinear forms on $\RR G$-modules, is central to our discussion.
\begin{lemma}
	\label{lem:bil}
	If $G$ is a finite group and $U$ is a finite-dimensional $\RR G$-module then $U$ has a non-degenerate, $G$-invariant,
	symmetric, bilinear form $B:U\times U\rightarrow \RR$ satisfying $B(u,u)>0$ for all non-zero $u\in U$.
	Moreover, if $U$ is irreducible, then any $G$-invariant bilinear form on $U$ is a scalar multiple of $B$.
\end{lemma}
\begin{proof}
	Let $\langle \cdot,\cdot\rangle$ be any inner product on $U$, i.e., $\langle \cdot, \cdot \rangle$ is a symmetric, 
	non-degenerate, positive definite, bilinear form. Then $B:U\times U \rightarrow \RR$, defined by
	\[ B(u,v) = \frac{1}{|G|}\sum_{g\in G} \langle g\cdot u,g\cdot v\rangle,\]
	is $G$-invariant, and is a convex combination of non-degenerate, symmetric, positive definite, bilinear forms. 
	Hence, $B$ also has these properties. Now suppose $U$ is irreducible and $B'$ is another $G$-invariant bilinear form on $U$.
	Since $B(\cdot,\cdot)$ is positive definite, we can 
	choose $\lambda$ such that the subspace $\{v\in U: \lambda B(u,v) = B'(u,v)\;\;\textup{for all $u\in U$}\}$ is non-zero.
	(Here, we can take $\lambda$ to be any generalized eigenvalue for the pair of symmetric matrices representing the bilinear forms.)
	Since $B$ and $B'$ are both $G$-invariant, this is a $G$-invariant subspace of $U$. Since $U$ is irreducible, 
	we must have $\{v\in U: \lambda B(u,v) = B'(u,v)\;\;\textup{for all $u\in U$}\} = U$. Hence $\lambda B = B'$, as required.
\end{proof}

We now record basic facts about $G$-invariant elements of tensor products.
\begin{lemma}
	\label{lem:ineq-tensor}
	If $G$ is a finite group and $U$ and $W$ be non-isomorphic irreducible finite-dimensional $\RR G$-modules,
	then $(U\otimes W)^G = \{0\}$.
\end{lemma}
\begin{proof}
	Suppose $W$ and $U$ are irreducible and not isomorphic.
	By Schur's lemma (see, for instance,~\cite[Section 2.2]{SerreBook}), $\textup{Hom}_{G}(W,U) = \{0\}$. 	
	Since $U$ is a finite-dimensional $\RR G$-module, by Lemma~\ref{lem:bil} it has a non-degenerate, $G$-invariant, bilinear form
	$B: U\times U\rightarrow \RR$. Hence the map $U\ni u \mapsto B(u,\cdot)\in U^*$ gives an isomorphism between $U$ and $U^*$. Then
	$(U\otimes W)^G \cong (U^* \otimes W)^G \cong \textup{Hom}_G(U,W) = \{0\}$.
\end{proof}
The following result tells us that if $U$ is irreducible, then $(U\otimes U)^G$ is one-dimensional. It also explicitly describes the map
$\sym:U\otimes U\rightarrow (U\otimes U)^G$.
\begin{lemma}
	\label{lem:sym-tensor}
	If $G$ is a finite group and $U$ is a finite-dimensional irredcuible $\RR G$-module then for all $s\in U$ and all
	$u,u'\in U$,
	\[ B(s,s)\,\sym(u\otimes u') = B(u,u')\, \sym(s\otimes s)\]
	where $B(\cdot,\cdot)$ is the (unique up to scale) non-zero $G$-invariant bilinear form on $U$.
\end{lemma}
\begin{proof}
	If $s=0$ then the statement clearly holds. Assume $s\neq 0$ and let $\eta\in (U\otimes U)^*$ be an arbitrary linear functional on $U\otimes U$.
	Then the map $(u,u')\mapsto \eta(\sym(u\otimes u'))$ is a $G$-invariant bilinear form on $U$. By Lemma~\ref{lem:bil},
	$\eta(\sym(u\otimes u')) = \kappa(\eta)B(u,u')$ for some $\kappa(\eta)\in \RR$. By putting $u=u'=s$, we see that
	$\kappa(\eta) = \eta(\sym(s\otimes s))/B(s,s)$. Together these imply that
	\[ B(s,s)\,\eta(\sym(u\otimes u')) = B(u,u')\, \eta(\sym(s\otimes s)).\]
 	Since $\eta$ was arbitrary, it follows that $B(s,s)\,\sym(u\otimes u') = B(u,u')\, \sym(s\otimes s)$ for all $u,u'\in U$.
\end{proof}

\subsection{Symmetrizing sums of squares}
\label{sec:app-symsq}
Let $\RR[{\mathsf{x}}]$ be the polynomial ring in $q$ indeterminates and let $G$
be a finite group acting linearly on $\RR^q$. Then $G$ acts on $\RR[{\mathsf{x}}]$
via $(g\cdot \mathsf{p})({\mathsf{x}}) = \mathsf{p}(g\cdot {\mathsf{x}})$.
If $\cI$ be an ideal invariant under the action of $G$, then the action of $G$ descends
to the quotient $\RR[{\mathsf{x}}]/\cI$.
Let $V$ be a finite-dimensional $G$-invariant subspace of $\RR[{\mathsf{x}}]/\cI$.
Let $V\otimes V$ be the tensor product (over $\RR$) of $V$ with itself.
Given $\mathsf{f}_1,\,\mathsf{f}_2\in V$, their product $\mathsf{f}_1\mathsf{f}_2\in \RR[{\mathsf{x}}]/\cI$
is a $\RR$-bilinear map $V\times V\rightarrow \RR[{\mathsf{x}}]/\cI$. Hence there is a linear map
$M_V: V\otimes V\rightarrow \RR[{\mathsf{x}}]/\cI$
such that
\[ \mathsf{f}_1\mathsf{f}_2 = M_V(\mathsf{f}_1\otimes \mathsf{f}_2).\]
The actions of $G$ on $V\otimes V$ and on $\RR[{\mathsf{x}}]/\cI$ are such that $M_V$ is a $\RR G$-module homomorphism.

As before, let $(S_{\lambda})_{\lambda\in \Lambda}$ be an enumeration of inequivalent
irreducible $\RR G$-modules.
Let $V = \bigoplus_{\lambda\in \Lambda} V_\lambda$ be the isotypic decomposition of $V$.
We now show that if we take functions from two different isotypic components of $V$ and
symmetrize their product, the result is zero.
\begin{lemma}
\label{lem:sym-prod-zero}
If $\lambda\neq \mu$, $\mathsf{f}_\lambda \in V_\lambda$, and $\mathsf{f}_\mu \in V_\mu$, then
$\sym(\mathsf{f}_\lambda \mathsf{f}_\mu) = 0$.
\end{lemma}
\begin{proof}
	First, let $\phi_{\mu}\in \textup{Hom}_{G}(S_\mu,V)$ and $\phi_\lambda\in \textup{Hom}_G(S_\lambda,V)$, and 
	let $u_\mu\in S_\mu$ and $u_{\lambda}\in S_{\lambda}$. Then, since $M_V$ and $\phi_\mu\otimes \phi_\lambda$ are both $\RR$-linear 
	and commute with the action of $G$, we have that
	\[ \sym(\phi_\mu(u_\mu)\phi_\lambda(u_\lambda)) = \sym(M_V((\phi_\mu\otimes \phi_\lambda)(u_\mu\otimes u_\lambda))) = 
	M_V((\phi_{\mu}\otimes \phi_\lambda)(\sym(u_\mu\otimes u_\lambda))).\]
	Since $\mu\neq \lambda$, Lemma~\ref{lem:ineq-tensor} tells us that $\sym(u_\mu\otimes u_\lambda)\in (S_\mu\otimes S_\lambda)^G = \{0\}$ and 
	so $\sym(\phi_\mu(u_\mu)\phi_\lambda(u_\lambda)) = 0$. Finally, because any $\mathsf{f}_\mu\in V_{\mu}$ is a 
	linear combination of elements of the form $\phi_\mu(u_\mu)$ for $\phi_{\mu}\in \textup{Hom}_{G}(S_\mu,V)$
	and $u_\mu\in S_\mu$ (and similarly for any $\mathsf{f}_\lambda\in V_{\lambda}$), we have that $\sym(\mathsf{f}_\mu\mathsf{f}_\lambda)=0$ 
	by bilinearity.
\end{proof}

We now investigate what happens when we symmetrize products of certain elements
of $V_{{\lambda}}$.
\begin{lemma}
	\label{lem:sym-prod-iso}
	Let $\phi,\psi\in \textup{Hom}_G(S_\lambda,V)$ and $u,u'\in S_\lambda$. Let $B$ be the unique
	(up to scale), $G$-invariant bilinear form on $S_\lambda$. Then, for any $s_\lambda\in S_{\lambda}$,
	\[ B(s_\lambda,s_\lambda)\,\sym(\phi(u)\psi(u')) = B(u,u')\,\sym(\phi(s_\lambda)\psi(s_\lambda)).\]
\end{lemma}
\begin{proof}
	Since $M_V$ and $\phi\otimes \psi$ are both $\RR$-linear and commute with the action of $G$,
	\[ \sym(\phi(u)\psi(u')) = \sym(M_V((\phi\otimes \psi)(u\otimes u'))) = M_V((\phi\otimes \psi)(\sym(u\otimes u'))).\]
	By substituting $u=u'=s_\lambda$, we have
	\[ \sym(\phi(s_\lambda)\psi(s_\lambda)) = M_V((\phi\otimes \psi)(\sym(s_\lambda\otimes s_\lambda))).\]
	Relating $\sym(u\otimes u')$ and $\sym(s_\lambda\otimes s_\lambda)$ via Lemma~\ref{lem:sym-tensor}, and
	using the fact that $M_V$ and $\phi\otimes \psi$ are both $\RR$-linear,
	we obtain the stated result.
\end{proof}

Our next aim is to describe the structure of $\sym(\mathsf{f}^2)$ for $\mathsf{f}\in V_{\bm{\lambda}}$.

\begin{proposition}
	\label{prop:sym-prod-iso-general}
Let $\phi_1,\phi_2,\ldots,\phi_{m_\lambda}$ be a basis for $\textup{Hom}_{G}(S_\lambda,V)$, let $s_\lambda\in S_\lambda$ be non-zero, and let
\[ Y^\lambda_{j\ell} = \sym(\phi_j(s_\lambda)\phi_\ell(s_\lambda))\]
for all $j,\ell\in [m_\lambda]$. If $\mathsf{f}\in V_\lambda$, then there exists a $m_\lambda\times m_\lambda$ psd matrix $Q^\lambda$ such that
\[ \sym(\mathsf{f}^2) = \sum_{j,\ell\in [m_\lambda]} Q^\lambda_{j\ell} Y^\lambda_{j\ell}.\]
\end{proposition}
\begin{proof}
	Let $\mathsf{f}\in V_{\lambda} = \{ \phi(u): \phi\in \textup{Hom}_G(S_\lambda,V),\; u\in S_\lambda\}$. Then
	a straightforward argument shows that there are $u_1,u_2,\ldots,u_{m_\lambda}\in S_\lambda$ such that
	\[ \mathsf{f} = \sum_{k\in [m_\lambda]} \phi_i(u_i).\]
	We expand $\sym(\mathsf{f}^2)$ in terms of $\sym(\phi_i(u_i)\phi_j(u_j))$ and apply Lemma~\ref{lem:sym-prod-iso}
	 to get
	\begin{equation*}
	 \sym(\mathsf{f}^2)  = \sum_{i,j\in [m_\lambda]} \sym(\phi_i(u_i)\phi_j(u_j)) =  \sum_{i,j\in [m_\lambda]}
	\frac{B(u_i,u_j)}{B(s_\lambda,s_\lambda)}\,\sym(\phi_i(s_\lambda)\phi_j(s_\lambda)).
	\end{equation*}
	Since $u\mapsto B(u,u)$ is a positive definite quadratic form and $s_\lambda \neq 0$, if we define $Q_{ij}^\lambda := B(u_i,u_j)/B(s_\lambda,s_\lambda)$,
	we see that $Q^\lambda$ is psd. Since $Y_{ij}^\lambda = \sym(\phi_i(s_\lambda)\phi_j(s_\lambda))$,
	we have obtained an expression for $\sym(\mathsf{f}^2)$ in the required form.
\end{proof}
We have seen how to symmetrize products from different isotypic components of $V$ and how to symmetrize squares of polynomials
from the same isotypic component. We can now combine these earlier arguments in a straightforward way to
desribe symmetry-reduced non-negativity certificates for sos that are invariant under the action of $G$.
This is the main result of the appendix (which is essentially the main result of~\cite{GatermannParrilo}).
\begin{theorem}
	\label{thm:GP-app}
	Let $V$ be a finite-dimensional $G$-invariant subspace of $\RR[{\mathsf{x}}]/\cI$
	with isotypic decomposition $V = \bigoplus_{\lambda\in \Lambda} V_\lambda$
	and corresponding multiplicities $m_\lambda$. For each $\lambda\in \Lambda$, fix a non-zero element $s_\lambda \in S_\lambda$.
	Let $\mathsf{b}_1,\ldots,\mathsf{b}_{m_\lambda}$ be a basis for the subspace $W_{s_\lambda}$. Define for each $\lambda\in \Lambda$
	\[ Y_{ij}^\lambda = \sym(\mathsf{b}_i\mathsf{b}_j)\]
	for $i,j\in [m_\lambda]$.
	Suppose $\mathsf{p}\in \RR[{\mathsf{x}}]/\cI$ is invariant under the action of $G$ and is $V$-sos.
	Then there exist $m_\lambda\times m_\lambda$ psd matrices $Q^\lambda$ such that
	\[ \mathsf{p} = \sum_{\lambda\in \Lambda} \sum_{i,j\in [m_\lambda]} Q_{ij}^\lambda Y_{ij}^\lambda.\]
\end{theorem}
\begin{proof}
	Fix $\lambda\in \Lambda$. By Lemma~\ref{lem:mult-iso}, the map $\phi \mapsto \phi(s_\lambda)$ gives an isomorphism between $W_{s_\lambda}$
	and $\textup{Hom}_G(S_\lambda,V)$. Hence, for each $i\in [m_\lambda]$, there is $\phi_i\in \textup{Hom}_G(S_\lambda,V)$ such that
	$\mathsf{b}_i = \phi_i(s_\lambda)$. Moreover, the $\phi_i$ for $i\in [m_\lambda]$ form a basis for $\textup{Hom}_G(S_\lambda,V)$.
	As such,
	\[ Y_{ij}^\lambda = \sym(\mathsf{b}_i\mathsf{b}_j) = \sym(\phi_i(s_\lambda)\phi_j(s_\lambda))\]
	for $i,j\in [m_\lambda]$.
	Now suppose that $\mathsf{p}$ is $G$-invariant and $V$-sos.
	Then, since $V = \bigoplus_\lambda V_\lambda$, we have that $\mathsf{f}_k = \sum_{\lambda\in\Lambda}  \mathsf{f}_{k\lambda}$ where
	each $\mathsf{f}_{k\lambda}\in V_{\lambda}$. Hence
	\[ \mathsf{p} = \sum_{k} \sum_{\lambda,\mu \in \Lambda} \mathsf{f}_{k\lambda}\mathsf{f}_{k\mu}.\]
	From Lemma~\ref{lem:sym-prod-zero}, we have that $\sym(\mathsf{f}_{k\lambda}\mathsf{f}_{k\mu}) = 0$ whenever $\lambda \neq \mu$.
	Using the fact that $\mathsf{p}$ is fixed by the action of $G$,
	\[ \mathsf{p} = \sym(\mathsf{p}) = \sum_k \sum_{\lambda,\mu\in \Lambda} \sym(\mathsf{f}_{k\lambda}\mathsf{f}_{k\mu}) =  \sum_k \sum_{\lambda \in \Lambda}
	\sym(\mathsf{f}_{k\lambda}^2)\]
	where each $\mathsf{f}_{k\lambda}\in V_\lambda$. Applying Proposition~\ref{prop:sym-prod-iso-general} for each $\lambda$ and each $k$,
	we see that there exist psd matrices $Q^{\lambda,k}$ such that
	\[ \mathsf{p} = \sum_{k}\sum_{\lambda \in \Lambda} \sum_{i,j\in [m_\lambda]}Y^\lambda_{ij} Q^{\lambda,k}_{ij} =
	\sum_{\lambda \in \Lambda}\sum_{j,\ell\in [m_\lambda]} Y^\lambda_{ij}\left(\sum_k Q^{\lambda,k}_{ij}\right).\]
	Defining $Q^\lambda = \sum_k Q^{\lambda,k}$, which is again psd, gives an expression for $\mathsf{p}$ as
	\[ \mathsf{p} = \sum_{\lambda\in \Lambda} \sum_{i,j\in [m_\lambda]}Y^\lambda_{ij}Q^\lambda_{ij} \]
	as we require.
\end{proof}

\subsection{Specialization to Theorem~\ref{thm:GP}}

This specializes to give Theorem~\ref{thm:GP} as follows. Let $q = \binom{n}{2}$ and let
$\cI_n = \langle \mathsf{x}_{ij}^2 = \mathsf{x}_{ij}\;\;1\leq i<j\leq n\rangle$ denote the square-free ideal in $\RR[{\mathsf{x}}]$.
Let $\RR[{\mathsf{x}}]/\cI_n = \RR[\cV_n]$ be the corresponding quotient of the polynomial ring. Let $G = \fS_n$ and suppose that $\fS_n$ acts on monomials by
$\fs\cdot \mathsf{x}_{ij} = \mathsf{x}_{\fs(i)\fs(j)}$. Let $V = \RR[\cV_n]_{\leq d}$
be the $\fS_n$-invariant subspace of square-free polynomials of degree at most $d$.

To relate Theorem~\ref{thm:GP} to Theorem~\ref{thm:GP-app}, we need to show that the subspaces $W_{\tau_{\bm{\lambda}}}$ defined in Section~\ref{sec:prelim}
are isomorphic as vector spaces to $W_s$ for some non-zero $s\in S_{\bm{\lambda}}$, and hence to the multiplicity space $\textup{Hom}_{\fS_n}(S_{\bm{\lambda}},V)$ 
(via Lemma~\ref{lem:mult-iso}). 
The argument requires the following fact, which follows from Young's rule (see, e.g.,~\cite[Theorem 4.6]{RST} for a statement of Young's rule in the required form)
and the fact that the diagonal Kostka numbers $K_{\bm{\lambda}\bm{\lambda}}$ are one~(see, e.g.,~\cite[Example 2.11.4]{SaganBook}).
\begin{lemma}
If $\tau_{\bm{\lambda}}$ is a tableau of shape $\bm{\lambda}$, then $\dim(S_{\bm{\lambda}}^{\fR_{\tau_{\bm{\lambda}}}})=1$.
\end{lemma}
\begin{lemma}
\label{lem:mult-iso-2}
	Fix a standard tableau $\tau_{\bm{\lambda}}$ of shape $\bm{\lambda}$. If
	 $s_{\tau_{\bm{\lambda}}}\in S_{\bm{\lambda}}^{\fR_{\tau_\lambda}}$ is non-zero, then
	\[ V_{\bm{\lambda}}^{\fR_{\tau_{\bm{\lambda}}}} = W_{\tau_{\bm{\lambda}}} = W_{s_{\tau_{\bm{\lambda}}}} =
	\{\phi(s_{\tau_{\bm{\lambda}}}): \phi\in \textup{Hom}_{\fS_n}(S_{\bm{\lambda}},V)\}.\]
\end{lemma}
\begin{proof}
	If $w\in W_{s_{\tau_{\bm{\lambda}}}}$ then there is $\phi\in \textup{Hom}_{\fS_n}(S_{\bm{\lambda}},V)$ such that $\phi(s_{\tau_{\bm{\lambda}}}) = w$.
	If $\fs\in \fR_{\tau_{\bm{\lambda}}}$ then
	\[ \fs\cdot w = \fs \cdot \phi(s_{\tau_{\bm{\lambda}}}) = \phi(\fs \cdot s_{\tau_{\bm{\lambda}}}) = \phi(s_{\tau_{\bm{\lambda}}}) = w.\]
 	Hence $W_{s_{\bm{\lambda}}}\subseteq V_{\lambda}^{\fR_{\tau_{\bm{\lambda}}}} = W_{\tau_{\bm{\lambda}}}$. On the other hand,
	if $w\in W_{\tau_{\bm{\lambda}}}$, then there are $v_1,\ldots,v_{m_{\bm{\lambda}}}\in S_{\bm{\lambda}}$ and
	$\phi_{1},\ldots,\phi_{m_{\bm{\lambda}}}$ such that
	$w = \sum_{i\in [m_{\bm{\lambda}}]} \phi_i(v_i)\in V_{\bm{\lambda}}$ and $w = \sym_{\fR_{\tau_{\bm{\lambda}}}}(w)$. Then
	\[ w = \sym_{\fR_{\tau_{\bm{\lambda}}}}(w) = \sum_{i\in [m_{\bm{\lambda}}]} \sym_{\fR_{\tau_{\lambda}}}(\phi_i(v_i))
	= \sum_{i\in [m_{\bm{\lambda}}]} \phi_i(\sym_{\fR_{\tau_{\lambda}}}(v_i)).\]
	Since $S_{\bm{\lambda}}^{\fR_{\tau_{\bm{\lambda}}}}$ is one-dimensional, it is spanned by $s_{\tau_{\bm{\lambda}}}$. 
	Hence for any $v_i\in S_{\bm{\lambda}}$, we have that $\sym_{\fR_{\tau_{\bm{\lambda}}}}(v_i)$
	is a scalar multiple of $s_{\tau_{\bm{\lambda}}}$. It then follows that $w \in W_{s_{\tau_{\bm{\lambda}}}}$, giving the reverse inclusion.
\end{proof}

Applying Theorem~\ref{thm:GP-app} in this setting, and choosing $s_{\bm{\lambda}}:= s_{\tau_{\bm{\lambda}}}$ (defined in Lemma~\ref{lem:mult-iso-2}) 
for each $\bm{\lambda}$, yields the statement of Theorem~\ref{thm:GP}.

\newpage

% !TEX root =  main.tex
\section{Ramsey Number $R(3,3)$}
\label{app:Ramsey sos}

In this section, we prove the following lemma.
\begin{lemma} \label{lem:sos proof of Ramsey}
\begin{align*}
\left(2-\sum_{2\leq i  \leq 6} \mathsf{x}_{1i}\right)^2 + \left(2-\sum_{2\leq i  \leq 6} (1-\mathsf{x}_{1i}) \right)^2 &\equiv 
-1 \mod \mathcal{I}.
\end{align*}
\end{lemma}

Recall that the ideal in Section~\ref{sec:ramsey} was
\begin{align*}
\mathcal{I} = \langle \mathsf{x}_e^2-\mathsf{x}_e \,\, \forall  e\in E(K_6) \rangle & +  \langle \mathsf{x}_{ij}\mathsf{x}_{ik}\mathsf{x}_{jk} \,\,\forall i<j<k\in [6] \rangle\\
 & +  \langle (1-\mathsf{x}_{ij})(1-\mathsf{x}_{ik})(1-\mathsf{x}_{jk}) \,\, \forall i<j<k\in [6] \rangle.
\end{align*}

The following fact saves us from duplicating all of our arguments.

\begin{lemma}\label{bluered}
$\mathsf{p}((\mathsf{x}_{ij})_{1\leq i < j \leq 6}) \in \mathcal{I}$ if and only if $\mathsf{p}((1-\mathsf{x}_{ij})_{1\leq i < j \leq 6}) \in \mathcal{I}$.
\end{lemma}

\proof
This follows by exchanging the colors of the edges. This can also be observed algebraically by working with the ideal.
\qed

\medskip
We now note that claws are in the ideal, allowing us to relate certain degree two and degree one elements of the ideal.

\begin{lemma}\label{claw}
For $i\in [6]$, and $j<k<l \in [6]\backslash \{i\}$,
$$\mathsf{x}_{ij}\mathsf{x}_{ik}\mathsf{x}_{il} \in \mathcal{I} \textup{ and } (1-\mathsf{x}_{ij})(1-\mathsf{x}_{ik})(1-\mathsf{x}_{il}) \in \mathcal{I}.$$

Consequently,

$$1-\sum_{2\leq i \leq 6} \mathsf{x}_{1i} + \sum_{2\leq i <j \leq 6} \mathsf{x}_{1i}\mathsf{x}_{1j} \in \mathcal{I}.$$

\end{lemma}

\proof Indeed, combinatorially, if $\mathsf{x}_{ij}\mathsf{x}_{ik}\mathsf{x}_{il}$ were one for some $i,j,k,l$, then edges $\{i,j\}$, $\{i,k\}$, and $\{i,l\}$ would be colored red. Then if we color any of $\{j,k\}$, $\{k,l\}$ and $\{j,l\}$ red, we would have a red triangle spanned by that edge and vertex $i$, but if we color all of them blue, then we would have a blue $j,k,l$-triangle. We can observe the same thing symmetrically for blue edges.

Algebraically, observe that $\mathsf{x}_{jk}\mathsf{x}_{jl}\mathsf{x}_{kl} \in \mathcal{I}$ and $(1-\mathsf{x}_{jk})(1-\mathsf{x}_{jl})(1-\mathsf{x}_{kl}) \in \mathcal{I}$ implies that
\begin{align*}
1 & \equiv \mathsf{x}_{jk}+\mathsf{x}_{jl}+\mathsf{x}_{kl} - (\mathsf{x}_{jk}\mathsf{x}_{jl} + \mathsf{x}_{jk}\mathsf{x}_{kl}+\mathsf{x}_{jl}\mathsf{x}_{kl}) + \mathsf{x}_{jk}\mathsf{x}_{jl}\mathsf{x}_{kl}\\
& \equiv \mathsf{x}_{jk} + \mathsf{x}_{jl}+\mathsf{x}_{kl} - (\mathsf{x}_{jk}\mathsf{x}_{jl} + \mathsf{x}_{jk}\mathsf{x}_{kl}+\mathsf{x}_{jl}\mathsf{x}_{kl}) \textup{ mod } \mathcal{I}.
\end{align*}
Hence
\begin{align*}
\mathsf{x}_{ij}\mathsf{x}_{ik}\mathsf{x}_{il}(1)& \equiv \mathsf{x}_{ij}\mathsf{x}_{ik}\mathsf{x}_{il}\left(\mathsf{x}_{jk} + \mathsf{x}_{jl}+\mathsf{x}_{kl} - (\mathsf{x}_{jk}\mathsf{x}_{jl} + \mathsf{x}_{jk}\mathsf{x}_{kl}+\mathsf{x}_{jl}\mathsf{x}_{kl})\right)\\
& \equiv 0 \textup{ mod } \mathcal{I}
\end{align*}
since every monomial on the right hand side contains either $\mathsf{x}_{ij}\mathsf{x}_{ik}\mathsf{x}_{jk}$ or $\mathsf{x}_{ij}\mathsf{x}_{il}\mathsf{x}_{jl}$ or $\mathsf{x}_{ik}\mathsf{x}_{il}\mathsf{x}_{kl}$, each of which is in the ideal. Thus $\mathsf{x}_{ij}\mathsf{x}_{ik}\mathsf{x}_{il} \in \mathcal{I}$. By Lemma \ref{bluered}, $(1-\mathsf{x}_{ij})(1-\mathsf{x}_{ik})(1-\mathsf{x}_{il}) \in \mathcal{I}$ as well.

Since $(1-\mathsf{x}_{1i})(1-\mathsf{x}_{1j})(1-\mathsf{x}_{1l}) \in \mathcal{I}$ whenever $1<i<j<k \leq 6$, we have that
$$0\equiv \prod_{2\leq i \leq 6} (1-\mathsf{x}_{1i}) \textup{ mod } \mathcal{I}.$$

Expanding the right hand side and using the fact that $\mathsf{x}_{1i}\mathsf{x}_{1j}\mathsf{x}_{1k} \in \mathcal{I}$ whenever $1<i <j<k \leq 6$, we obtain
$$0 \equiv 1 - \sum_{2\leq i \leq 6} \mathsf{x}_{1i} + \sum_{2\leq i<j \leq 6} \mathsf{x}_{1i}\mathsf{x}_{1j} \textup{ mod } \mathcal{I}$$
since all of the higher order terms must vanish.
\qed

\medskip

\noindent{\em Proof of Lemma~\ref{lem:sos proof of Ramsey}}:
\begin{align*}
\left(2-\sum_{2\leq i  \leq 6} \mathsf{x}_{1i}\right)^2 & = 4 -4\sum_{2\leq i  \leq 6} \mathsf{x}_{1i}+\sum_{2\leq i \leq 6} \mathsf{x}_{1i}^2 +2\sum_{2\leq i<j\leq 6} \mathsf{x}_{1i}\mathsf{x}_{1j}\\
& \equiv 4 -4\sum_{2\leq i  \leq 6} \mathsf{x}_{1i}+\sum_{2\leq i \leq 6} \mathsf{x}_{1i} +2(-1 + \sum_{2\leq i \leq 6}x_{1i}) \quad \mod \mathcal{I} \\
&= 2-\sum_{2\leq i \leq 6} \mathsf{x}_{1i}.
\end{align*}

By Lemma \ref{bluered}, $\left(2-\sum_{2\leq i  \leq 6} (1-\mathsf{x}_{1i})\right)^2 \equiv 2-\sum_{2\leq i \leq 6} (1-\mathsf{x}_{1i})$.
Therefore,
\begin{align*}
\left(2-\sum_{2\leq i  \leq 6} \mathsf{x}_{1i}\right)^2 + \left(2-\sum_{2\leq i  \leq 6} (1-\mathsf{x}_{1i}) \right)^2 &\equiv 2 -\sum_{2\leq i \leq 6} \mathsf{x}_{1i} + 2 - \sum_{2\leq i  \leq 6} (1-\mathsf{x}_{1i}) \\
&= 4 - 5 = -1 \mod \mathcal{I}.
\end{align*}
\qed

For completeness, we explicitly wrote out how we used the ideal in the above sum of squares, putting the elements of the ideal in red. These calculations can be found online: http://www.math.washington.edu/$\sim$thomas/ramsey$\_$calculations.pdf.

\end{appendix}

\end{document}